\documentclass{article}

\usepackage{polski}
\usepackage[english]{babel}
\usepackage[utf8]{inputenc}
\usepackage{amsthm}
\usepackage{amsmath}
\usepackage{amsfonts}
\usepackage{amssymb}
\usepackage{mathtools}
\usepackage{mathrsfs}
\usepackage[a4paper, left=3cm, right=3cm, top=3cm, bottom=3cm]{geometry}
\usepackage{url}
\usepackage[all]{xy}

\newcommand{\wt}[1]{{\widetilde{#1}}}
\providecommand{\dim}{\mathop{\rm dim}\nolimits}
\providecommand{\lf}{\mathop{\rm lf}\nolimits}
\providecommand{\deg}{\mathop{\rm deg}\nolimits}
\providecommand{\vol}{\mathop{\rm vol}\nolimits}

\providecommand{\Ricci}{\mathop{\rm Ricci}\nolimits}
\providecommand{\im}{\mathop{\rm im}\nolimits}

\providecommand{\supp}{\mathop{\rm supp}\nolimits}
\providecommand{\Lip}{\mathop{\rm Lip}\nolimits}
\providecommand{\Alt}{\mathop{\rm Alt}\nolimits}

\newtheorem{thm}{Theorem}[section]
\newtheorem{prop}[thm]{Proposition}
\newtheorem{lemma}[thm]{Lemma}
\newtheorem{cor}[thm]{Corollary}

\theoremstyle{definition}

\newtheorem{remark}[thm]{Remark}
\newtheorem{defi}[thm]{Definition}

\title{Lipschitz simplicial volume of connected sums}
\author{Karol Strzałkowski}

\begin{document}
\maketitle

\begin{abstract}
We prove that the locally finite simplicial volume and the Lipschitz simplicial volume are additive with respect to certain gluings of manifolds. In particular, we prove that in dimension $\geq 3$ they are additive with respect to connected sums and gluings along $\pi_1$-injective, amenable aspherical boundary components.
\end{abstract}

\section{Introduction}
The simplicial volume is a homotopy invariant of manifolds defined for a closed manifold $M$ as
\[
\|M\| :=inf\{|c|_1\::\: \text{$c$ is a fundamental cycle with $\mathbb{R}$ coefficients}\},
\]
where $|\cdot|_1$ is the $\ell^1$-norm on $C_*(M,\mathbb{R})$ (which we will denote for simplicity as $C_*(M)$) with respect to the basis consisting of singular simplices. In other words, it is an $\ell^1$-norm of the fundamental class. It can be also defined for manifolds with boundary by taking the $\ell^1$-norm of the relative fundamental class. Although the definition is relatively simple, it has many applications. Most of them are mentioned in the work of Gromov \cite{G}. One of the most important is the use to degree theorems. It is clear from the definition that the simplicial volume is functorial in the sense that if $f:M\rightarrow N$ is a map between $n$-dimensional manifolds then
\[
\deg(f)\cdot\|N\|\leq \|M\|.
\]
One obtains immediately that if $\|N\|\neq0$ then
\[
\deg(f)\leq \frac{\|M\|}{\|N\|}.
\]
It follows that we are particularly interested in the examples of manifolds with non-zero simplicial volume, because only in this case we get some non-trivial bounds on the degrees of maps. However, the existence of such bounds reveals some kind of rigidity of a given manifold and it is much easier to give examples of manifolds without such a rigid behaviour, i.e. with zero simplicial volume. These are e.g. all manifolds which admit a self-map of a degree $>1$ such as spheres and tori. There is also a beautiful result of Gromov that $\|M\|=0$ if $M$ has amenable fundamental group or admits a non-trivial circle action \cite{G}.

The simplest group of examples of manifolds with non-zero simplicial volume are closed surfaces of genus $\geq 2$, and more generally negatively curved manifolds. The main ingredients of the proof are the existence of the straightening procedure for such manifolds (which is valid for all $CAT(0)$-spaces and allow to change every singular chain into a chain with not greater $\ell^1$-norm consisting of geodesic simplices) and the upper bound on the volume of geodesic simplices. Similar strategy leads to the discovery of other manifolds with non-zero simplicial volume, such as locally symmetric spaces of non-compact type \cite{LafS, Sav}.

Another way to obtain manifolds with non-zero simplicial volume is to construct them from such manifolds by certain operations. Two of the operations which may be used for such construction are products and connected sums
\begin{thm}[\cite{G}]\label{Theorem_product_inequality}
  Let $M$ and $N$ be two compact manifolds. Then the following inequality holds
 \[
 \|M\|\cdot\|N\| \leq \|M\times N\| \leq \binom{\dim M +\dim N}{\dim M} \|M\|\cdot\|N\|.
 \]
\end{thm}
\begin{thm}[\cite{G}]\label{Theorem_commected_sums_classic}
  Let $M$ and $N$ be two compact manifolds of dimension $n\geq 3$. Then
  \[
  \|M\#N\| = \|M\| +\|N\|.
  \]
\end{thm}

In fact, the second of the above theorems is a special case of additivity with respect to amenable gluings \cite{G, BBFIPP, Kue1}.

\begin{thm}
Let $M_1$, $M_2$ be two compact manifolds with boundary. Let $N_i\subset M_i$ for $i=1,2$ be boundary components such that there exists a homeomorphism $f:N_1\rightarrow N_2$. Assume moreover that $im(\pi_1(N_1)\rightarrow \pi_1(M_1))$ and $im(\pi_1(N_1)\rightarrow \pi_1(M_1))$ are amenable and
\[
f_*(ker(\pi_1(N_1)\rightarrow \pi_1(M_1))) = ker(\pi_1(N_2)\rightarrow \pi_1(M_2)).
\]
Then
\[
\|M_1\cup_f M_1, \partial(M_1\cup_f M_2)\| = \|M_1, \partial M_1\| + \|M_2, \partial M_2\|.
\]
\end{thm}
In particular, simplicial volume is additive with respect to gluings along $\pi_1$-injective boundary components with amenable fundamental groups.

For non-compact manifolds there are several ways of generalising the simplicial volume. The simplest and most intuitive approach is to define it as the $\ell^1$ norm of the locally finite (relative) fundamental class. The resulting simplicial volume, which we will call the \emph{locally finite simplicial volume} and denote also by $\|\cdot\|$, is invariant under proper homotopy equivalences. However, this volume vanishes in many cases \cite[Section 4.2, Example (a)]{G} and some theorems that hold for the simplicial volume in the compact case do not hold for the locally finite simplicial volume. The examples are product inequality (Theorem \ref{Theorem_product_inequality}) or proportionality principle \cite[Section 0.4]{G}. Another possible approach is to include a metric in the definition of the simplicial volume. This philosophy is realised e.g. by the \emph{Lipschitz simplicial volume}
\[
\|M\|_{\Lip}:= inf\{|c|_1\::\: c\in C_n^{\lf,\Lip}(M) \text{ is a fundamental cycle}\}.
\]
Here, $H^{\lf,\Lip}_*(M)$ is locally finite Lipschitz homology, i.e. homology of the locally compact Lipschitz chain complex
\[
C_*^{\lf,\Lip} = \{c = \sum_i a_i\sigma_i\in C_*^{\lf}(M)\::\: \Lip(c) = \sup_i \Lip(\sigma_i)<\infty  \}.
\]
The Lipschitz simplicial volume equals the classical simplicial volume for compact manifolds and is more rigid invariant than locally finite simplicial volume for non-compact manifolds-it is invariant under proper Lipschitz homotopy equivalences. On the other hand, the Lipschitz simplicial volume allows us to generalise most of the theorems concerning the simplicial volume to the non-compact case, including product inequality and proportionality principle \cite{LS, KS, Fra}.

One can ask about the behaviour of both of the above versions of the simplicial volume under taking certain gluings. In fact, the topic has not been much studied. The only reference known to the author is the classic work of Gromov \cite{G}, where the following version of the additivity for the locally finite simplicial volume is stated.

\begin{thm}[\cite{G}, Section 4.2]
Let $\dim M>1$ and let $V = \cup_{j=0}^{\infty}\ V_j \subset M$ be a disjoint union of compact submanifolds. If $im(\pi_1(V_j)\rightarrow\pi_1(M))$ are amenable for all $j\in\mathbb{N}$ and the sequence $(V_j)_{j\in\mathbb{N}}$ is amenable at infinity, then
\[
\|M\setminus V\| \geq \|M\|.
\]
Moreover, if $V$ consists of closed, codimension-$1$ submanifolds such that $\pi_1(V_j)$ injects in $\pi_1(M)$ for $j\in\mathbb{N}$ then
\[
\|M\setminus V\| = \|M\|.
\]
\end{thm}
Here, a sequence $(V_j)_{j\in\mathbb{N}}$ is \emph{amenable at infinity} if there is a descending sequence $(U_j)_{j\in\mathbb{N}}$ of subsets of $M$ such that $\cap_{j=1}^{\infty}U_j=\emptyset$ and for every $k\in\mathbb{N}$ one has $\cup_{j=k}^{\infty}V_j\subset U_k$ and $\im(\pi_1(V_k)\rightarrow \pi_1(U_k))$ are amenable. However, Gromov did not prove the above theorem, but only indicated how should it be proven. Using Gromov's hints, it is clear how to prove the first part, but to the author's knowledge the proof of the second part remains unknown. Even less is known for the Lipschitz simplicial volume, for which there are no results on the additivity.

One of the reasons why so little is known about the additivity of simplicial volumes for non-compact manifolds in contrast with the compact case is that all known proofs of such additivity phenomena use at some point duality principle between the simplicial volume and bounded cohomology. Although this principle in the compact case is easily stated and succesfuly applied to many problems, it is usually much more complicated for simplicial volumes for non-compact manifolds. In the case of the locally finite simplicial volume author believes that one can use the existing results to prove the additivity with respect to connected sums. However, this kind of proof would most probably not apply to the Lipschitz simplicial volume, where the duality principle is even more complicated.

In this paper we present the first geometric proofs of certain additivity theorems concerning the locally finite and Lipschitz simplicial volumes. Namely, we prove the following.

\begin{thm}\label{Theorem_main}
Let $M_1$ and $M_2$ be two $n$-dimensional Riemannian manifolds with boundary and let $N_i\subset \partial M_i$ for $i=1,2$ be some compact boundary components such that there exists homeomorphism $f:N_1\rightarrow N_2$. Then if the group $\im(\pi_1(N_1)\rightarrow \pi_1(M_1\cup_f M_2))$ is amenable then
\[
\|M\cup_f N, \partial (M\cup_f N)\| \leq \|M_1,\partial M_1\| + \|M_2,\partial M_2\|.
\]
Moreover, if the groups $\pi_1(N_1)$ and $\pi_1(N_2)$ are amenable and inject into $\pi_1(M_1)$ and $\pi_1(M_2)$ respectively and one of the following conditions is satisfied
\begin{itemize}
\item $N_1$ is aspherical;
\item $\pi_k(N_1)=0$ for $k=2,...,n-2$ and for every connected component $N'\subset \partial(M_1\cup_f M_2)$ the group $\im(\pi_1(N')\rightarrow \pi_1(M_1\cup_f M_2))$ is amenable;
\end{itemize}
then
\[
\|M\cup_f N, \partial (M\cup_f N)\| = \|M_1,\partial M_1\| + \|M_2,\partial M_2\|.
\]
If $f$ is bi-Lipschitz, the same holds also for $\|\cdot\|_{\Lip}$.
\end{thm}

\begin{cor}
If $M_1$ and $M_2$ have no boundary and $n\geq 3$ then
\[
\|M_1\# M_2\| = \|M_1\| + \|M_2\|
\]
and
\[
\|M_1\# M_2\|_{\Lip} = \|M_1\|_{\Lip} + \|M_2\|_{\Lip}.
\]
\end{cor}

The geometric nature of the proof allows us to prove the above statement for most other versions of simplicial volume described by Gromov in \cite{G}. For simplicity we will prove Theorem \ref{Theorem_main} only for the locally finite simplicial volume, but we explain in Section \ref{section_prelims} how to easily generalise the proof to other versions of simplicial volume.

Theorem \ref{Theorem_main} allows us to generalise some degree theorems. Note that a connected sum of two manifolds $M_1$ and $M_2$ is well defined up to homeomorphism, but not in the world of Riemannian manifolds. In particular, the volume of $M_1\# M_2$ is not well defined, because it depends on the volume of discs which are cut out from $M_1$ and $M_2$ and on the volume of the glued cylinder $S^{n-1}\times I$, which may be any value in $\mathbb{R}_+$. Depending on these, the volume of $M_1\# M_2$ may vary from some value smaller than $M_1\# M_2$ to $+\infty$. In the following, we use the convention that the connected sum $M_1\# M_2$ (which is still not well defined as a Riemannian manifold) satisfy the additivity rule with respect to volume, i.e.
\[
\vol(M_1\# M_2) = \vol(M_1)+\vol(M_2),
\]
assuming $M_1$ and $M_2$ are of dimension $\geq 3$ (in the case $\dim M_i=1$ for $i=1,2$ $M_i$ are circles and there is nothing interesting to prove, in the case $\dim =2$ manifolds $M_1$, $M_2$ and $M_1\# M_2$ are surfaces, which usually come with metrics of constant curvatures and do not satisfy the above additivity with respect to both simplicial and Riemannian volume)

\begin{thm}
Let $\mathcal{C}$ be the smallest class of manifolds containing locally symmetric spaces of non-compact type (with standarized metrics) of finite volume and closed under taking products and connected sums. For every $n\in\mathbb{N}$ there exists a constant $D_n\in\mathbb{R}$ such that if $M$ and $N$ are $n$-dimensional Riemannian manifolds satisfying $|\sec(M)|\leq 1$, $\Ricci(M)\geq -(n-1)$ and $N\in\mathcal{C}$ then for every proper Lipschitz map $f:M\rightarrow N$ we have
\[
|\deg f| \leq D_n\frac{\vol(M)}{\vol(N)}.
\]
\end{thm}

\begin{proof}
Note that by \cite[Theorem 1.8]{LS} we have $\|M\|_{\Lip}\leq B_n\cdot\vol(M)$ for some constant $B_n$ depending only on $n$. Therefore we have
\[
|\deg f|\leq \frac{\|M\|}{\|N\|} \leq B_n\frac{\vol(M)}{\|N\|}.
\]
We will prove that there exists a constant $C_n$ such that
\[
\|N\|_{\Lip} \geq C_n\cdot \vol(N).
\]
by the double induction on $n$ and on the 'complexity' $k$ of $N$, i.e. minimal number of times we need to take products or connected sums of locally symmetric spaces of non-compact type in order to obtain $N$. For $n=1$ there are no $1$-dimensional manifolds in $\mathcal{C}$ and for $n=2$, as well as for any $n\in\mathbb{N}$ and 'complexity' $k=0$ the result follows from the result of L\"{o}h and Sauer \cite{LS}. Assume the theorem is true for $n<N$ and let $C'_N$ be a constant for which the theorem holds for 'complexity' $k=0$. We will show by induction on $k$ that the theorem holds in general for
\[
C_N = \min\{C'_N, C_{n_1}\cdot C_{n_2}\::\: n_1+n_2 = N\,,\, n_1,n_2<N\}.
\]
If $N = N_1\times N_2$, then
\[
\|N\|_{\Lip} \geq \|N_1\|_{\Lip}\times \|N_2\|_{\Lip} \geq C_{n_1}\cdot C_{n_2}\cdot \vol(N_1\times N_2) \geq C_N\vol(N)
\]
On the other hand, if $N= N_1\# N_2$ then
\[
\|N\|_{\Lip} = \|N_1\|_{\Lip} + \|N_2\|_{\Lip} \geq C_N\cdot (\vol(N_1)+\vol(N_2)) = C_N\cdot\vol(N).
\]
This finishes the proof.
\end{proof}

\subsection*{Organization of this work}
In Section \ref{section_prelims} we fix the notation, introduce some basic lemmas and explain the technique allowing us to generalize all the proofs to various geometric general simplicial volumes. In Section \ref{section_subadditivity} we prove the subadditivity of the simplicial volume with respect to certain gluings and introduce some machinery useful in the proof of superadditivity. In particular, in Sections \ref{subsection_delta_sets}, \ref{subsection_automorphisms} and \ref{subsection_PI} we reintroduce Gromov's machinery of multicomplexes and adapt it to our situation, while in Section \ref{subsecton_local_barycentric} we introduce piecewise barycentric subdivision. In Section \ref{section_amenable_trick} we prove Theorem \ref{Theorem_main} for gluings along aspherical boundary components. Finally, in Section \ref{section_higher_dim_trick} we introduce 'higher dimensional cell trick' which allows to generalize a little bit the result from the previous section and finish the proof of Theorem \ref{Theorem_main}.

\section{Preliminaries}\label{section_prelims}
\subsection{Notation}\label{subsection_notation}
Throughout this paper we will often modify geometrically singular chains and simplices, therefore we need to clarify notation and recall some basic facts. In the following section, $X$ is a topological space.

For $k\in\mathbb{N}$, we treat a simplex $\Delta^k$ as a metric subspace of $\mathbb{R}^{k+1}$ defined as
\[
\Delta^k = \{(x_0,...,x_k)\in\mathbb{R}_+^k\::\: \sum_{i=0}^k x_i = 1\}.
\]
For $i=0,...,k$ we denote also by $\delta^i:\Delta^{k-1}\rightarrow \Delta^k$ the standard embeddings onto $i$-th face of $\Delta^k$:
\[
\delta^i(x_0,...,x_{k-1}) = (x_0,...,x_{i-1}, 0, x_i,...,x_{k-1}).
\]
Note that there is a natural action of the symmetric group $\Sigma_{k+1}$ on $\Delta^k$ by permuting the coordinates. This action induces an action of $\Sigma_{k+1}$ on the set of singular simplices $C(\Delta^k, X)$ by
\[
(\pi \cdot \sigma)(x) = \sigma(\pi\cdot x)
\]
where $\pi\in\Sigma_{k+1}$, $\sigma\in C(\Delta^k, X)$ and $x\in \Delta^k$.

Let $\sigma\in C(\Delta^k, X)$ be a singular simplex. $\sigma$ is therefore formally a map. However, in many cases we will be interested rather in the image of $\sigma$ than in the map itself. Therefore we will use $\sigma$ to denote both the map and its image, (e.g. for $Y\subset X$ we will denote by $\sigma\subset Y$ the fact that $\im(\sigma)\subset Y$). It should be clear from the context which of these two meanings we use. The same applies to vertices, edges and higher-dimensional faces of $\sigma$, which are formally maps, but we will use their symbols also to denote their images in $X$.

Let $c=\sum_i a_i\sigma_i\in C_k^{\lf,\Lip}(X)$ be a locally finite singular chain. It can be considered as a map $C(\Delta^k,X)\rightarrow\mathbb{R}$, hence we will sometimes use the corresponding notation, e.g. $c(\sigma)$ to denote a coefficient of $\sigma$ in a chain $c$ or $\supp(c)$ as the set of simplices with non-zero coefficients in $c$. In particular,
\[
c = \sum_{\sigma\in C(\Delta^k,X)} c(\sigma)\cdot\sigma = \sum_{\sigma\in \supp(c)} c(\sigma)\cdot\sigma.
\]

By the $l$-skeleton of $c$ we will understand all $l$-faces of all simplices $\sigma\in\supp(c)$. We will denote it by $c^{(l)}$. In particular, $c^{(k)}=\supp(c)$.

We will denote by $|c|_1$ the $\ell^1$-norm of a chain $c$ and by $\|[c]\|_1$ the $\ell^1$-semi-norm of its homology class (if $c$ is a cycle). Moreover, let $S\subset C(\Delta^k, X)$ be some subset of singular $k$-simplices We will denote by $|c|^S_1$ the $\ell^1$-norm of $c$ counted on the simplices from $S$, i.e.
\[
|c|^S_1 := \sum_{\sigma\in S}|c(\sigma)|.
\]
We will be particularly interested in three types of such subsets $S$. For $Y\subset X$, we define
\begin{itemize}
\item $n(Y) := C(\Delta^k, X)\setminus C(\Delta^k, Y)$ (the set of simplices not contained in $Y$);
\item $e(Y):= \{\sigma\in C(\Delta^k, X)\::\: \sigma \text{ has some edge in }Y\}$;
\item $ne(Y) := C(\Delta^k, X)\setminus e(Y) = \{\sigma\in C(\Delta^k, X)\::\: \sigma \text{ has no edges in }Y\}$.
\end{itemize}

\subsection{Chain homotopies}\label{subsection_chain_homotopies}
Given a chain $c\in C_k^{\lf}(X)$, we will often try to modify it without increasing its $\ell^1$-norm. The basic fact allowing us to perform such operations is the following.

\begin{lemma}\label{lemma_basic_chain_homotopy}
Let $N\in\mathbb{N}$ and let $H^{\sigma}_k:C(\Delta^k,X)\times I \rightarrow C(\Delta^k,X)$ for $\sigma\in C(\Delta^k,X)$ and $k<N$ be a system of homotopies such that for every $k<N$, $\sigma\in C(\Delta^k, X)$, and $i=0,...,k$,
\[
H^{\sigma}_k|_{\partial_i\Delta^k\times I} = H^{\partial_i\sigma}_{k-1}.
\]
Then for every chain $c \in C_k(X)$ for $k<N$, the chain
\[
c' = \sum_{\sigma\in C(\Delta^k, X)} c(\sigma)\cdot H^{\sigma}_k(\cdot, 1)
\]
is chain homotopic to $c$ and $|c'|_1\leq |c|_1$.
\end{lemma} 

The proof is standard and is described e.g. in \cite[Proof of Theorem 2.10]{HAT} or \cite[Lemma 2.13]{LS}. However, the above lemma cannot be used for locally finite and Lipschitz chains without some additional assumptions, which are not always satisfied. On the other hand, in most cases a local modification of a given chain would suffice.

\begin{lemma}\label{lemma_local_chain_homotopy}
Let $N\in\mathbb{N}$ and let $H^{\sigma}_k: C(\Delta^k,X)\times I \rightarrow C(\Delta^k,X)$ for $k<N$ and $\sigma\in C(\Delta^k,X)$ be a system of homotopies such that for every $k<N$, $\sigma\in C(\Delta^k, X)$, and $i=0,...,k$ we have
\[
H^{\sigma}_k|_{\partial_i\Delta^k\times I} = H^{\partial_i\sigma}_{k-1}.
\]
Then for every compact subset $Y\subset X$ and every cycle $c \in C^{\lf}_k(X)$ for $k<N$, there is a cycle $c' = \sum_{\sigma\in C(\Delta^k, X)} c(\sigma)\cdot\tau_{\sigma} \in C^{\lf}_k(X)$ homologuous to $c$ such that
\begin{enumerate}
\item $\tau_{\sigma} = \sigma$ for almost all $\sigma\in \supp(c)$;
\item $\tau_{\sigma} = H_k^{\sigma}(\cdot, 1)$ for every $\sigma\in\supp(c)$ such that $\sigma \cap Y\neq \emptyset$;
\item $|c'|_1\leq |c|_1$.
\end{enumerate}
\end{lemma}

\begin{proof}
Let $Y' = \cup_{\sigma\in\supp(c)\::\: \sigma\cap Y\neq\emptyset} \im(\sigma)$ and let $f:X\rightarrow\mathbb{R}$ be a compactly supported function such that $f|_{Y'} \equiv 1$. Define 
\[
c' = \sum_{\sigma\in C(\Delta^k, X)} c(\sigma)\cdot(f*H^{\sigma}_k),
\]
where $f*H^{\sigma}_k\in C(\Delta^k, X)$ for $k<N$ and $\sigma\in C(\Delta^k, X)$ is defined as
\[
(f*H^{\sigma}_k)(x) = H^{\sigma}_k(x, f(x)).
\]
Checking that $c'$ satisfies all the required properties is straightforward.
\end{proof}

The first application of the above lemma is the proposition that allows us to generalize most the proofs in this paper to the Lipschitz case. However, before we state it, we need some more terminology.

\begin{defi}
Let $c, c'\in C^{\lf}_k(X)$ be two chains and let $Y\subset X$. We say that $c'$ is a \emph{finite modification} of $c$ in $C^{\lf}_k(X,Y)$ if 
\[
(\Sigma_{k+1}\cdot\supp(c))\Delta (\Sigma_{k+1}\cdot\supp(c'))
\]
is finite, where $\Delta$ is the symmetric difference operator. In other words, $c$ and $c'$ have almost the same support relative to $Y$ up to the action of $\Sigma_{k+1}$. We say also that $c'$ is a \emph{finite approximation} of $c$ if it is finite modification of $c$ and is homologuous to $c$.
\end{defi}

\begin{defi}
We say that $X$ has the \emph{Lipschitz simplicial approximation property} (LSA) if every singular simplex $\sigma\in C(\Delta^k, X)$ is homotopic to some Lipschitz singular simplex $\sigma'\in \Lip(\Delta^k, X)$ and if $\sigma|_{\partial\Delta^k}$ is a Lipschitz map, then we can assume this homotopy is constant on $\partial\Delta^k\times I$.
\end{defi}

\begin{prop}\label{prop_generalization_Lipschitz}
Assume that $X$ has LSA. Let $c\in C^{\lf,\Lip}_*(X,Y)$ be a cycle and let $c'\in C^{\lf}_*(X,Y)$ be its finite modification (in $C_*^{\lf}(X,Y)$). Then there exists a finite approximation $c''\in C_*^{\lf,\Lip}(X,Y)$ of $c'$ such that $|c''|_1\leq |c'|_1$.
\end{prop}

\begin{proof}
First, using property LSA we construct a system of homotopies $H^{\sigma}_k$ for $\sigma\in C(\Delta^k, X)$ and $k\in\mathbb{N}$ satisfying the assumptions of Lemma \ref{lemma_basic_chain_homotopy} by induction on $k$. For $k=0$ we choose constant homotopies, and if $H^{\sigma}_l$ are defined for $\sigma\in C(\Delta^l, X)$ and $l=0,...,k-1$ then we define $H^{\sigma}_k$ for $\sigma\in C(\Delta^k, X)$ simply as a homotopy extension of $H^{\sigma}_k|_{\partial\Delta^k\times I}$ such that $H^{\sigma}_k(\cdot, 1)$ is Lipschitz. Then we apply Lemma \ref{lemma_local_chain_homotopy} to this system, cycle $c'$ and $Y= \cup_{\sigma\in(\Sigma_{k+1}\cdot\supp(c))\Delta(\Sigma_{k+1}\cdot\supp(c'))} \im(\sigma)$ and obtain a cycle $c''$ satisfying the required conditions.
\end{proof}

\begin{remark}\label{remark_generalization}
In the rest of this work, for simplicity we will state and prove the results only for the locally finite simplicial volume. However, the main statements are proved by constructing a suitable cycle being a finite modification/approximation of the Lipschitz one. It is also easy to observe that every Riemannian manifold has LSA by Whitney approximation theorem \cite[Theorem 6.19]{LeeS}. Therefore the actual proof of \ref{Theorem_main} for the Lipschitz simplicial volume follows easily from the above proposition and the remaining part of this paper.
\end{remark}

\section{Amenability and subadditivity of the $\ell^1$-norm}\label{section_subadditivity}
This section is devoted to the proof of the subadditivity of the simplicial volume with respect to certain gluings, i.e. the first part of Theorem \ref{Theorem_main}. In fact, we prove a little bit more general fact.

\begin{prop}\label{prop_subadditivity}
Let $(X_1,Y_1)$ and $(X_2,Y_2)$ be two pairs of topological spaces and let $Z_i\subset Y_i$ for $i=1,2$ be two compact path-connected components such that there is a homeomorphism $f:Z_1\rightarrow Z_2$. Assume moreover that $\im(\pi_1(Z_1)\rightarrow \pi_1(X_1\cup_f X_2))$ is amenable. Let $k\in\mathbb{N}$ and let $c_1\in C^{\lf}_k(X_1, Y_1)$ and $c_2\in C^{\lf}_k(X_2, Y_2)$ be two cycles such that
\[
f_*([(\partial c_1)|_{Z_1}]) = -[(\partial c_2)|_{Z_2}]
\]
holds in $H_{k-1}(Z_2)$. Then for every $\varepsilon>0$ there exists a cycle $c\in C_k^{\lf}(X_1\cup_f X_2, (Y_1\setminus Z_1)\cup (Y_2\setminus Z_2))$ such that its restrictions $c'_1$ and $c'_2$ to $C_k^{\lf}(X_1\cup_f X_2, Y_1\cup X_2)$ and $C_k^{\lf}(X_1\cup_f X_2, X_1\cup Y_2)$ respectively are finite approximations of $c_1$ and $c_2$ respectively, and
\[
|c|_1\leq |c_1|_1 + |c_2|_1 +\varepsilon.
\]
In particular, $\|[c]\|_1 \leq \|[c_1]\|_1+ \|[c_2]\|_1$.
\end{prop}

\begin{remark}\label{remark_well_defined_homology_restrictions}
In the above proposition we compare $c'_1\in C^{\lf}_*(X_1\cup_f X_2, Y_1\cup X_2)$ and $c'_2\in C^{\lf}_*(X_1\cup_f X_2, X_1\cup Y_2)$ with $c_1\in C^{\lf}_k(X_1, Y_1)$ and $c_2\in C^{\lf}_k(X_2, Y_2)$, which may cause some confusion. However, note that there are obvious embeddings
\[
C^{\lf}_k(X_1, Y_1)\hookrightarrow C^{\lf}_*(X_1\cup_f X_2, Y_1\cup X_2)
\]
and
\[
C^{\lf}_k(X_2, Y_2)\hookrightarrow C^{\lf}_*(X_1\cup_f X_2, X_1\cup Y_2)
\]
and the notion of being finite modification/approximation is in this case independent of the space we are considering, since $Z_2$ is compact.
\end{remark}

The above proposition follows easily from the following fact.

\begin{prop}\label{prop_amenable_subset}
Let $(X,Y)$ be a pair of topological spaces and let $Z\subset X$ be a path-connected, relatively compact set such that $\im(\pi_1(Z)\rightarrow \pi_1(X))$ is amenable. Let $c\in C^{\lf}_k(X,Y)$ be a locally finite cycle. For every $\varepsilon>0$ there exists a finite approximation $c'$ of $c$ such that
\[
|c'|_1\leq |c|^{ne(Z)}_1 + \varepsilon.
\]
In particular,
\[
\|[c]\|_1 \leq |c|^{ne(Z)}_1.
\]
\end{prop}

\begin{proof}[Proof of Proposition \ref{prop_subadditivity}]
Let $\varepsilon>0$. By assumption, there exists a chain $d\in C_k(Z_2)$ such that
\[
\partial d = f_*((\partial c_1)|_{Z_1}) + (\partial c_2)|_{Z_2}.
\]
Consider the chain $c_3 := c_1+c_2 -d$. We compute that it is a cycle in $C_*^{\lf}(X_1\cup_f X_2, Y_1\cup Y_2)$.
\[
\partial c_3 = \partial c_1 +\partial c_2 - (\partial c_1)|_{Z_1} - (\partial c_2)|_{Z_2} = (\partial c_1)|_{Y_1\setminus Z_1} - (\partial c_2)|_{Y_2\setminus Z_2}\in C_*^{\lf}(Y_1\cup Y_2).
\]
Moreover, it is obvious that its restrictions to $C^{\lf}_*(X_1\cup_f X_2, Y_1\cup X_2)$ and $C^{\lf}_*(X_1\cup_f X_2, X_1\cup Y_2)$  are finite approximations of $c_1$ and $c_2$ respectively. Now it suffices to apply Proposition \ref{prop_amenable_subset} to $c_3$, $\varepsilon$ and $Z_2\subset X_1\cup_f X_2$ to obtain a cycle $c$ satisfying the required conditions.
\end{proof}

The proof of Proposition \ref{prop_amenable_subset} can be divided into two steps, which are represented by the following lemmas.

\begin{defi}
Let $(X,Y)$ be a pair of topological spaces, let $Z\subset X$ and let $\sigma\in C(\Delta^k,X)$ for $k\in\mathbb{N}$ be a singular simplex. We say it is \emph{$Z$-non-degenerated} if all the vertices of $\sigma$ which are contained in $Z$ are distinct. We say that a singular chain $c\in C_*^{\lf}(X,Y)$ is $Z$-non-degenerated if every $\sigma\in\supp(c)$ is $Z$-non-degenerated.
\end{defi}

\begin{lemma}\label{lemma_amenable_subset_distinct_vertices}
Let $(X,Y)$ be a pair of topological spaces and let $Z\subset X$ be a path-connected, relatively compact set such that $\im(\pi_1(Z)\rightarrow \pi_1(X))$ is amenable. Let $c\in C_k^{\lf}(X,Y)$ be a $Z$-non-degenerated cycle. Then for every $\varepsilon>0$ there exists a finite approximation $c'$ of $c$ such that
\[
|c'|_1\leq |c|^{ne(Z)}_1 + \varepsilon.
\]
In particular,
\[
\|[c]\|_1 \leq |c|^{ne(Z)}_1.
\]
\end{lemma}

The proof of Lemma \ref{lemma_amenable_subset_distinct_vertices} involves multicomplexes machinery and the diffusion of chains, techniques used by Gromov in \cite[Section 4.2]{G}. However, the version proved by Gromov was a little bit less general than the above one, therefore we need to complete the proof, which we do in Section \ref{subsection_proof_of_subadd} after preparations made in Sections \ref{subsection_delta_sets}, \ref{subsection_automorphisms} and \ref{subsection_PI}.
 
The second step of the proof of Proposition \ref{prop_amenable_subset} is the following lemma, which is proved using local barycentric subdivision, described in Section \ref{subsecton_local_barycentric}.

\begin{lemma}\label{lemma_local_barycentric}
Let $(X,Y)$ be a pair of topological spaces, let $Z\subset X$ and let $c\in C_k^{\lf}(X,Y)$ be a cycle. Then there is a $Z$-non-degenerated finite approximation $c'$ of $c$ such that
\[
|c'|^{ne(Z)}_1 \leq |c|^{ne(Z)}_1.
\]
\end{lemma}

\subsection{$\Delta$-sets and multicomplexes}\label{subsection_delta_sets}
In this section we introduce Gromov's machinery concerning multicomplexes, $\Delta$-sets and diffusion of chains. They are described e.g. in \cite{G, KKu}. However, in our approach we need to apply these techniques locally. Moreover, our space $X$ is not always aspherical. We introduce only the notions and facts that have some use for our purposes. More standard and complete approach can be found in the references given above.

\begin{defi}
A \emph{$\Delta$-set} $(S_*,V,\partial)$ with the set of vertices $V$ is the following data
\begin{itemize}
\item Family of sets $S_{v_0,...,v_k}$ for each $k\in\mathbb{N}$ and each ordered sequence $(v_0,...,v_k)\in V^{k+1}$. The set $S_{v_0,...,v_k}$ is called a set of simplices with vertices $v_0,...,v_k$.
\item Family of maps $\partial_i:S_{v_0,...,v_k}\rightarrow S_{v_0,...,\hat{v_i},...v_k}$ for each $S_{v_0,...,v_k}$ and $i=0,...,k$, such that
\[
\partial_i\partial_j = \partial_{j-1}\partial_i
\]
for $i< j$.
\end{itemize}
A $\Delta$-set is a \emph{multicomplex} if $S_{v_0,...,v_k}=\emptyset$ whenever $v_i=v_j$ for $i\neq j$.
\end{defi}

For a subset $W\subset V$, we will denote by $S_W$ the full sub-$\Delta$-set consisting of all simplices with vertices in $W$.

As for the simplicial set, we can define a \emph{geometric realisation} of $|S|$ as
\[
|S| := \bigcup_{k\in\mathbb{N}} \bigcup_{(v_0,...,v_k)\in V^k} S_{v_0,...,v_k}\times\Delta^k / \sim,
\]
where $\sim$ is an equivalence relation generated by
\[
(S_{v_0,...,v_k},\delta^i x)\sim (S_{v_0,...,\hat{v_i},...v_k},x)
\]
for $x\in \partial_i\Delta^k$. Note that if $W\subset V$, then $|S_W|$ is a subset of $|S_V|$. Note also that for any map $f: K\rightarrow L$ of $\Delta$-sets there is a canonically defined map $|f|: |K|\rightarrow |L|$. On the other hand, we call a map $g: |K|\rightarrow |L|$ \emph{simplicial} if it is a geometric realization of some map $K\rightarrow L$. 

A very important example of a $\Delta$-complex, which we will have in mind, is the complex of singular simplices in a topological space $X$, denoted as $S_*(X)$. The set of vertices of $S_*(X)$ is obviously $X$.

\begin{defi}
Let $f:|K|\rightarrow |L|$ be a map between two $\Delta$-sets. We say that $f$ is an \emph{simplicial immersion} if it is simplicial and injective on the interior of every simplex $\sigma\in K$.
\end{defi}

\begin{defi}
Let $S$ be a $\Delta$-set with the vertex set $V$ and let $W\subset V$. It is \emph{$W$-non-degenerated} if $S_W$ is a multicomplex. A $W$-non-degenerated $\Delta$-set $S$ is
\begin{itemize}
\item \emph{complete} if every $|S_W|$-non-degenerated map $f:\Delta^{k}\rightarrow |S|$ such that $f|_{\partial\Delta^k}$ is a simplicial immersion is homotopic to a simplicial immersion relative to $\partial\Delta^k$;
\item \emph{$W$-locally minimal} if every $|S_W|$-non-degenerated singular simplex $f:\Delta^{k}\rightarrow |S|$ with some vertices in $W$ such that $f|_{\partial\Delta^k}$ is a simplicial immersion is homotopic to at most one simplicial immersion relative to $\partial\Delta^k$;
\end{itemize}
\end{defi}

Note that every complete $W$-non-degenerated $\Delta$-set $S$ contains a complete $W$-locally minimal $W$-non-degenerated $\Delta$-set. One can construct it simply by choosing inductively for each $k\in\mathbb{N}$ and each $\sigma\in S_{v_0,...,v_k}$ one element in its homotopy class relative to its boundary, whenever $\{v_0,...,v_k\}\cap W\neq\emptyset$.

For the following two lemmas we will need one more denotation. Let $\sigma\in C(\Delta^k, X)$ be a singular simplex and let $Z\subset X$. We denote by $\Sigma^Z_{\sigma}\subset \Sigma_{k+1}$ the set of those permutations which permute only the vertices of $\sigma$ which are contained in $Z$.

\begin{lemma}
There exists a $\Delta$-set $K_*(X)\subset S_*(X)$ with vertices $X$ such that
\begin{enumerate}
\item $K_*(X)$ is $Z$-locally minimal, complete $Z$-non-degenerated $\Delta$-set;
\item if $\sigma\in K_*(X)$ is homotopic relative to its boundary to a simplex in $Z$, then $\sigma\subset Z$;
\item for every $\sigma\in K_k(X)$ and every $s\in\Sigma^Z_{\sigma}$ we have $s\cdot \sigma\in K_k(X)$.
\end{enumerate}
\end{lemma}

\begin{proof}
Let $L_*(X)\subset S_*(X)$ be the $\Delta$-set containing all the simplices $\sigma\in S_*(X)$ such that all vertices of $\sigma$ contained in $Z$ are distinct. Then $L_*(X)$ is $W$-non-degenerated and complete. Therefore one can define $K_k(X)$ inductively on $k$ by setting $K_0(X) = X$ and choosing exactly one simplex from every homotopy class (relative to their boundaries) of simplices in $L_k(X)$ such that their boundaries lies in $K_{k-1}(X)$. Moreover, we can choose only the simplices satisfying the rest of the indicated conditions, using the fact that if $\sigma$ is $Z$-non-degenerated and $s\in\Sigma^Z_{k+1}$ is non-trivial then $s\cdot\sigma$ cannot be homotopic relative to its boundary to $\sigma$.
\end{proof}

\begin{defi}
We say that a locally finite singular chain $c$ is \emph{$Z$-locally $K_*(X)$-admissible} if for every $\sigma\in supp(c)$ such that $\sigma\cap Z\neq\emptyset$ we have $\sigma\in K_*(X)$.
\end{defi}

\begin{defi}
We say that a chain $c\in C_k^{\lf}(X,Y)$ for $k\in\mathbb{N}$ is \emph{$Z$-antisymmetric} if for every $\sigma\in\supp(c)$ and $s\in\Sigma^Z_{\sigma}$ we have
\[
c(s\cdot\sigma) = (-1)^{|s|}c(\sigma).
\]
\end{defi}

\begin{lemma}\label{lemma_Z-admissible}
Let $Z\subset X$ be relatively compact and let $c\in C_*^{\lf}(X,Y)$ be a $Z$-non-degenerated locally finite singular cycle. Then there is a $Z$-locally $K_*(X)$-admissible finite approximation $c'$ of $c$ such that $|c'|_1\leq |c|_1$ and $|c'|^{ne(Z)}_1\leq |c|_1^{ne(Z)}$. Moreover, if $c$ is $Z$-antisymmetric then we can choose $c'$ to also be $Z$-antisymmetric.
\end{lemma}

\begin{proof}
We will construct by induction on $k$ a family of homotopies $H^{\sigma}_k:\Delta^k\times I\rightarrow \Delta^k$ for $\sigma\in C(\Delta^k, X)$ and $k\in\mathbb{N}$ satisfying the conditions of Lemma \ref{lemma_local_chain_homotopy} such that $H^{\sigma}_0$ are constant homotopies and for every $k\in\mathbb{N}$ and $Z$-non-degenerated $\sigma\in C(\Delta^k, X)$ we have $H^{\sigma}_k(\cdot, 1)\in K_k(X)$. $H^{\sigma}_0$ are defined, assume that $H^{\sigma}_k$ are defined for $\sigma\in C(\Delta^k, X)$ and $k\leq N$. Using the completeness of $K_*(X)$, one can define $H^{\sigma}_{N+1}$ for $Z$-non-degenerated $\sigma$ as homotopy extensions of $H^{\sigma}_{N+1}|_{\partial\Delta^N}$ such that $H^{\sigma}_N(\cdot, 1)\in K_N(X)$ and for other $\sigma$ as any homotopy extensions of $H^{\sigma}_{N+1}|_{\partial\Delta^N}$. Using this system of homotopies we apply Lemma \ref{lemma_local_chain_homotopy} to $c$ and $\cup_{\sigma\in\supp(c)\::\: \sigma\cap Z\neq\emptyset}\im(\sigma)$ and obtain a cycle $c'$ with desired properties. Because we can also assume that in the above procedure simplices in $e(Z)$ stay in $e(Z)$, we have $|c'|^{ne(Z)}_1\leq |c|_1^{ne(Z)}$.

Moreover, we can choose $H^{\sigma}_k$ for $Z$-non-degenerated $\sigma\in C(\Delta^k,X)$ and $k\in\mathbb{N}$ such that they are $\Sigma^Z_{\sigma}$-equivariant, in the sense that for every $s\in\Sigma^Z_{\sigma}$ we have
\[
H^{s\cdot\sigma}_k(x,t) = H^{\sigma}_k(s\cdot x,t).
\]
Note that because $\sigma$ is $Z$-non-degenerated, $s\cdot\sigma\neq\sigma$, hence the above definition of $H^{s\cdot\sigma}_k(x,t)$ is correct.
\end{proof}

\subsection{Automorphisms of $\Delta$-sets}\label{subsection_automorphisms}
The main reason why we use $\Delta$-sets and multicomplexes is because their automorphisms can be easily constructed and behave in a nice way. The following lemmas are simple generalizations of the corresponding lemmas about multicomplexes from \cite{G}.

\begin{lemma}\label{lemma_multicomplex_automorphism}
Let $W\subset V$ and $W'\subset V'$. Assume that $(S,V,\partial)$ is $W$-non-degenerated, complete and $W$-locally minimal $\Delta$-set and $(S',V',\partial)$ is $W'$-non-degenerated, complete and $W'$-locally minimal $\Delta$-set. If $f:|S|\rightarrow |S'|$ is a simplicial homotopy equivalence which is bijective on vertices and which restriction $f|_{S_{V\setminus W}}:S_{V\setminus W}\rightarrow S'_{V'\setminus W'}$ is an isomorphism, then it is a bijection.
\end{lemma}

\begin{proof}
We will prove the bijectivity between $k$-skeletons of $S$ and $S'$ for $k\in\mathbb{N}$ inductively on $k$. We know that $f|_{S^{(0)}}:S^{(0)}\rightarrow S'^{(0)}$ is a bijection by assumption, so assume $f$ is a bijection on $k-1$ skeletons. To prove the injectivity of $f|_{S^{(k)}}$, let $\sigma, \tau\in S^{(k)}$ be two $k$-simplices with some vertices in $W$ such that $f(\sigma)=f(\tau)$ (if they do not have vertices in $W$, the thesis is obvious). By the bijectivity of $f|_{S^{(k-1)}}:S^{(k-1)}\rightarrow S'^{(k-1)}$ they have the same $k-1$-dimensional faces. Because the sphere $f(\sigma)\cup_{f(\partial\sigma)}f(\tau)$ is homotopically trivial and $f$ is a homotopy equivalence, the sphere $\sigma\cup_{\partial{\sigma}}\tau$ is homotopically trivial. Hence $\sigma=\tau$ by $W$-local minimality of $S$.

To show the surjectivity of $f|_{S^{(k)}}:S^{(k)}\rightarrow S'^{(k)}$, let $\sigma'\in S'^{(k)}$ be a $k$-simplex. Because its boundary (as a map $\partial\Delta^k\rightarrow |S'|$) is homotopically trivial, $f$ is a homotopy equivalence and $f|_{S^{(k-1)}}:S^{(k-1)}\rightarrow S'^{(k-1)}$ is bijective, there exists a simplex $\sigma\in S^{(k)}$ such that $\partial(f(\sigma)) = \partial\sigma'$ as maps $\partial\Delta^k\rightarrow |S'|$. Consider the sphere $\sigma'\cup_{\partial\sigma'}f(\sigma)$. Because $f$ is a homotopy equivalence and $S$ is complete, there exists a simplex $\sigma''\in S^{(k)}$ with the same boundary as $\sigma$ such that $f(\sigma''\cup_{\partial\sigma}\sigma)$ is homotopic to $\sigma'\cup_{\partial\sigma'}f(\sigma)$. Therefore $f(\sigma'')$ is homotopic to $\sigma'$ relative to their common boundary, hence by the $W$-local minimality of $S'$, $\sigma'=f(\sigma'')$.
\end{proof}

\begin{lemma}\label{lemma_multicomplex_extension}
Let $S$ be a $W$-non-degenerated complete $\Delta$-set and let $S'$ be a subcomplex containing $S_{V\setminus W}\cup W$. Then every simplicial map $f:|S'| \rightarrow |S|$ homotopic to the identity relative to $S_{V\setminus W}$ such that $f|_W:W\rightarrow W$ is a bijection can be extended to a (not unique in general) simplicial map $\tilde{f}:|S|\rightarrow |S|$ homotopic to the identity relative to $|S_{V\setminus W}|$.
\end{lemma}

\begin{proof}
To prove the first part, assume inductively that $\tilde{f}$ is defined on $S^{(k-1)}$ for some $k\in\mathbb{N}$. Let $\sigma\in S^{(k)}$ be a simplex with some vertices in $W$ and let $H:\partial\Delta^k\times I \rightarrow |S|$ be a homotopy between the identity on $\partial\sigma$ and $\tilde{f}(\partial\sigma)$. We extend this homotopy to a homotopy $\tilde{H}:\Delta^k\times I\rightarrow |S|$ such that $H(\cdot, 0) = \sigma$. Let $\sigma' = \tilde{H}(\cdot, 1)$. The boundary of $\sigma'$ is a subcomplex of $S$ by the inductive hypothesis, hence by completeness there is a simplex $\sigma''\in S^{(k)}$ which is homotopic to $\sigma'$ relative to its boundary. We define $\tilde{f}(\sigma) = \sigma''$. The fact that $\tilde{f}$ is homotopic to the identity is clear from the construction.
\end{proof}

Let $S$ be a $W$-non-degenerated complete $W$-minimal $\Delta$-set and let $\Gamma^W(S)$ be a group of simplicial automorphisms of $S$, constant on $S_{V\setminus W}$ and homotopic to the identity relative to $S_{V\setminus W}$. Let also $\Gamma^W_i(S)< \Gamma^W(S)$ be a subgroup consisting of the elements fixing $S^{(i)}$.

\begin{lemma}\label{lemma_multicomplex_group}
%Let $\sigma\in S^{(k)}$ be a simplex with some vertex in $W$ and let $S_{\sigma}$ be the set of all simplices in $S$ with the same boundary as $\sigma$ (as maps $\partial\Delta^{k-1}\rightarrow |S|$). Then $\Gamma^W_{k-1}$ acts transitively on $S_{\sigma}$. 
$\Gamma^W_1(S)/ \Gamma^W_k(S)$ is amenable for every $k\in\mathbb{N}_+$.
\end{lemma}

\begin{proof}
We will show, following Gromov in \cite{G}, that for every $i\in\mathbb{N}_+$ the group $\Gamma^W_i(S)/\Gamma^W_{i+1}(S)$ is abelian as it embeds in a product of $\pi_{i+1}(|S|)$. In the rest of the proof, given two maps $f,g:D^i\rightarrow X$ to a space $X$ such that $f|_{S^{i-1}} = g|_{S^{i-1}}$ we will denote by $[f,g]$ the map $S^i\rightarrow X$ defined as $f$ on the upper hemisphere and $g$ on the lower one and by $[[f,g]]$ an element of $\pi_i(X)$ represented by $[f,g]$. In particular, $[[f,g]]^{-1} = [[g,f]]$ and for any map $h:D^i\rightarrow X$ such that $h|_{S^{i-1}} = f|_{S^{i-1}} = g|_{S^{i-1}}$ we have
\[
[[f,g]] = [[[f,h],[g,h]]],
\]
where we use the same notation for $f,g:S^i\rightarrow X$, treating them as functions defined on $D^i$ and constant on $S^{i-1}$. Note that using this notation, if $f,g:S^i\rightarrow X$ then $[[f,g]] = [f]\cdot ([g])^{-1}$ in $\pi_i(X)$.

Let $\sigma\in S^{(i+1)}$ be a simplex. Consider a map $\rho_{\sigma}:\Gamma^W_i(S)\rightarrow \pi_{i+1}(|S|, x_{\sigma})$, where $x_{\sigma}\in\partial_0\sigma$, defined as $\rho_{\sigma}(f):= [[f\circ\sigma, \sigma]]$. We claim that this is a homomorphism. For $f,g\in \Gamma^W_i(S)$ we have
\begin{eqnarray*}
\rho(f\circ g) & = & [[f\circ g\circ\sigma, \sigma]] \\ 
& = & [[[f\circ g\circ\sigma, f\circ\sigma],[\sigma, f\circ\sigma]]] \\
& = & [[[g\circ\sigma, \sigma],[\sigma, f\circ\sigma]]] \\
& = & [[g\circ\sigma, \sigma]]\cdot([[\sigma, f\circ\sigma]])^{-1} \\
& = & [[g\circ\sigma, \sigma]]\cdot [[f\circ\sigma, \sigma]] = \rho(g)\cdot \rho(f),
\end{eqnarray*}
where in the third equation we used the fact that $f$ is homotopic to the identity relative to $S^{(k-1)}$.

Now, consider a homomorphism
\[
\prod_{\sigma\in S^{(i+1)}}\rho_{\sigma}:\Gamma^W_i(S)\rightarrow \prod_{\sigma\in S^{(i+1)}}\pi_{i+1}(|S|,x_{\sigma}).
\]
Its kernel is $\Gamma^W_{i+1}(S)$. Indeed, if $[f(\sigma), \sigma]$ is zero in $\pi_{i+1}(|S|,x_{\sigma})$ for $\sigma\in S^{(i+1)}$ and $f\in\Gamma^W_i(S)$, then by the local $W$-minimality of $|S|$ we have $\sigma = f(\sigma)$, hence $f$ fixes $\sigma$.
\end{proof}

\subsection{Group $\Pi(X,Z)$}\label{subsection_PI}
In this section we come back to more 'concrete' setting, so let $X$ be a metric space and $Z\subset X$ its path-connected compact subset.

Let $\Pi(X,Z)$ be the group consisting of families $([\gamma_x])_{x\in Z}$ of homotopy classes (in $X$) of paths in $Z$ such that
\begin{enumerate}
\item $\gamma_x(0) = x$ for all $x\in Z$;
\item $\gamma_x$ is constant for all but finitely many $x\in Z$; 
\item the map $x\mapsto \gamma_x(1)$ is a bijection of $Z$.
\end{enumerate}
The group structure is given by concatenation of paths, namely
\[
((\gamma_x)_{x\in Z}\cdot (\gamma'_x)_{x\in Z} )_{x_0} = \gamma_{x_0} * \gamma'_{\gamma_{x_0}(1)}.
\]

There is a right action of $\Pi(X,Z)$ on the $1$-skeleton of $K_*(X)$ such that for an edge $e\in K_1(X)$ the edge $e\cdot (\gamma_x)_{x\in Y}$ is a unique edge in $K_1(X)$ which is homotopic relative to its endpoints to $\gamma^{-1}_{e(0)}* e* \gamma_{e(1)}$. Because $\Pi(X,Z)$ acts by automorphisms fixing edges with endpoints in $X\setminus Z$, by Lemmas \ref{lemma_multicomplex_automorphism} and \ref{lemma_multicomplex_extension} every element of $\Pi(X,Z)$ can be extended to an element of $\Gamma^Z(K_*(X))$.

In the following lemma, for a simplex $\sigma\in K_*(X)$ we denote by $\Upsilon^Z_{\sigma}<\Sigma^Z_{\sigma}$ the group generated by all transpositions of vertices that are joined by an edge in $Z$.

\begin{lemma}\label{lemma_properties_Gamma}
Let $\Gamma(X,Z)< \Gamma^Z(K_*(X))$ be the subgroup of all possible extensions of elements of $\Pi(X,Z)$ and let $\Gamma_k(X,Z) = \Gamma(X,Z)/ \Gamma^Z_k(K_*(X))$ for $k\in\mathbb{N}_+$. Then
\begin{enumerate}
\item $\Gamma_k(X,Z)$ preserves $K_*(X)\cap e(Z)$ and $K_*(X)\cap ne(Z)$;
\item for every $l=0,...,k$ and $\sigma\in K_l(X)$ with some vertex in $Z$, $\Gamma_k(X,Z)$ acts transitively on $\Upsilon^Z_{\sigma}\cdot\sigma$;
\item if $\im(\pi_1(Z)\rightarrow \pi_1(X))$ is amenable then $\Gamma_k(X,Z)$ is amenable.
\end{enumerate}
\end{lemma}

\begin{proof}
\leavevmode
\begin{enumerate}
\item This part is obvious.
\item Let $K^{\sigma}_*(X)\subset K_*(X)$ be a sub-$\Delta$-set made form $(K_*(X))_{X\setminus Z}\cup Z$ and $\sigma$ with all its faces and let $s\in\Upsilon^Z_{\sigma}$. Consider a map $f:K^{\sigma}_*(X)\rightarrow K_*(X)$ sending $\sigma$ to $s\cdot\sigma$ and being the identity on $(K_*(X))_{X\setminus Z}$ and $Z\setminus (Z\cap\sigma^{(0)})$. It clearly induces a bijection on $Z$. Moreover, it is homotopic to the identity embedding relative to $(K_*(X))_{X\setminus Z}$. To see this, let $H_s:\Delta^l\times I\rightarrow \Delta^l$ be a geodesic homotopy joining $Id_{\Delta^l}$ with $s$, i.e.
\[
H_s(x,t) = [x, s\cdot x](t).
\]
Then a homotopy $H:K^{\sigma}_*(X)\times I\rightarrow K_*(X)$ between $Id_{K^{\sigma}_*(X)}$ and $f$ is given by
\[
H(x, t):= \begin{cases}
x & \text{ for }x\in (K_*(X))|_{X\setminus Z}\cup (Z\setminus (Z\cap\sigma^{(0)})); \\ 
\sigma\circ H_s(x,t) & \text{ for } x\in\sigma.
\end{cases}
\]
This homotopy is well defined, because it is constant on the faces of $\sigma$ that have no vertices in $Z$ and if some faces of $\sigma$ are the same, the above homotopy preserves these identifications (we use here the fact that $\sigma$ is $Z$-non-degenerated). By Lemma \ref{lemma_multicomplex_extension} $f$ can be extended to a map $\tilde{f}: K_*(X)\rightarrow K_*(X)$ homotopic to the identity and by Lemma \ref{lemma_multicomplex_automorphism} $\tilde{f}$ is an automorphism. Finally, it is easy to see that $\tilde{f}$ is an extension of the following element $(g_s)_{x\in Z}\in \Pi(X,Z)$:
\[
(g_s)_x(t) = \begin{cases}
x & \text{ for } x\notin \sigma^{(0)}; \\
e_{x, s\cdot x}(t) & \text{ for } x\in\sigma^{(0)},
\end{cases}
\]
where $e_{x,y}$ for $x,y\in\sigma^{(0)}$ is an (oriented) edge joining $x$ and $y$. The above element is well defined because $\sigma$ is $Z$-non-degenerated.

\item Note that we have an exact sequence
\[
0 \rightarrow \Gamma^Z_1(K_*(X))/\Gamma^Z_k(K_*(X)) \rightarrow \Gamma_k(X,Z) \rightarrow \Pi(X,Z) \rightarrow 0.
\]
Because $\Gamma^Z_1(K_*(X))/\Gamma^Z_k(K_*(X))$ is amenable by Lemma \ref{lemma_multicomplex_group}, it suffices to show that $\Pi(X,Z)$ is amenable. However, we have an exact sequence
\[
\xymatrix{
0 \ar[r] & \bigoplus_{x\in Z}(im(\pi_1(Z)\rightarrow \pi_1(X))) \ar[r]  & \Pi(X,Z) \ar[r] & \Sigma^{fin}(Z) \ar[r] & 0,
}
\]
where $\Sigma^{fin}(Z)$ is the group of finitely supported permutations of the set $Z$, hence the amenability of $\Pi(X,Z)$ follows from the amenability of $\bigoplus_{x\in Z}(im(\pi_1(Z)\rightarrow \pi_1(X)))$ and $\Sigma^{fin}(Z)$.
\end{enumerate}
\end{proof}

Note that $\Gamma_k(X,Z)$ acts isometrically on the set of $Z$-locally $K_*(X)$-admissible $k$-chains in the following way. If $g\in\Gamma_k(X,Z)$ and $c\in C_k^{\lf}(X,Y)$ is $Z$-locally $K_*(X)$-admissible, then
\[
g\cdot c = \sum_{\sigma\in \supp(c)} c(\sigma)\cdot (g\cdot\sigma),
\]
where we use the convention that $g\cdot\sigma = \sigma$ if $\sigma\notin K_k(X)$. Note that because there are only finitely many simplices in $\supp(c)$ which intersect $Z$, the chain $g\cdot c$ is a finite modification of $c$. Moreover, because the automorphisms in $\Gamma_k(X,Z)$ are homotopic to the identity, if $c$ is a cycle then $[g\cdot c] = [c]$, hence $g\cdot c$ is a finite approximation of $c$.

\subsection{Proof of Lemma \ref{lemma_amenable_subset_distinct_vertices}}\label{subsection_proof_of_subadd}
We are almost ready to prove Lemma \ref{lemma_amenable_subset_distinct_vertices}. We need only two more ingredients. The first one is the following proposition, proved in \cite{FuMa} for singular chains, but the proof for locally finite chains is exactly the same.

\begin{prop}\label{prop_antisymmetrisation}
Let $\Alt:C_*^{\lf}(X,Y)\rightarrow C_*^{\lf}(X,Y)$ be an antisymmetrisation operator defined for $c\in C_k^{\lf}(X,Y)$ as
\[
\Alt(c)  = \frac{1}{|\Sigma_{k+1}|}\sum_{\sigma\in\supp(c)}\sum_{s\in\Sigma_{k+1}} (-1)^{|s|}c(\sigma)(s\cdot\sigma).
\]
Then $\Alt$ is a chain map chain homotopic to the identity.
\end{prop}
An important consequence of the above proposition is that $Alt(c)$ is a finite approximation of $c$. Note also that by definition $|\Alt(c)|^{ne(Z)}_1\leq |c|^{ne(Z)}_1$. The second ingredient is the following proposition, proved by Gormov in \cite{G}, and a simple corollary. Recall that a \emph{probability measure} on a discrete group $G$ is a non-negative element of $\ell^1(G)$ of norm $1$.

\begin{prop}[{\cite[Section 4.2]{G}}]\label{prop_Gromov_diffusion}
Let $A$ be a set, let $f:A\rightarrow \mathbb{R}$ be a finitely supported function and let $G$ be an amenable group acting transitively on $A$. Then for every $\varepsilon>0$ there exists a finitely supported probability measure $\mu$ on $G$ such that
\[
\|\mu*f\|_1\leq \varepsilon + |\sum_{a\in A}f(a)|,
\]
where $(\mu*f)(x) = \sum_{g\in G}\mu(g)f(g^{-1}x)$.
\end{prop}

\begin{cor}\label{corollary_Gromov_diffusion}
Let $f:A\rightarrow \mathbb{R}$ be a finitely supported function and let $G$ be an amenable group acting on $A$ and let $A_1, ..., A_N\subset X$ be the orbits of the elements of $\supp(f)$. Then for every $\varepsilon>0$ there exists a finitely supported probability measure $\mu$ on $G$ such that
\[
\|\mu*f\|_1\leq \varepsilon + \sum_{i=1}^N |\sum_{a\in A_i}f(a)|
\]
\end{cor}

\begin{proof}
Consider $f_i = f|_{A_i}\in \ell^1(A_i)$ for $i=1,...,N$. In particular, $\|f\|=\sum_{i=1}^N\|f_i\|_1$. By induction, choose finitely supported probability measures $\mu_k\in \ell^1(G)$ for $k=1,...,N$, such that
\[
\|\mu_k*(\mu_{k-1}*...*\mu_1*f_k)\|_1\leq \frac{\varepsilon}{N} + |\sum_{a\in A_k}\mu_{k-1}*...\mu_1*f_k(a)| = \frac{\varepsilon}{N} + |\sum_{a\in A_k}f_k(a)|.
\]
Then for $\mu = \mu_N*...*\mu_1$, we have
\begin{eqnarray*}
\|\mu*f\|_1 & = & \|\mu_N*...*\mu_1*f\|_1 = \sum_{i=1}^N \|\mu_N*...\mu_1*f_i\|_1 \\
& \leq & \sum^{N}_{i=1}\|\mu_N\|_1\cdot...\|\mu_{i+1}\|_1\cdot \|\mu_i*...\mu_1*f_i\|_1 \\
& \leq & \sum^{N}_{i=1}(\frac{\varepsilon}{N} + |\sum_{a\in A_i}f_i(a)|) = \varepsilon + \sum_{i=1}^N |\sum_{a\in A_i}f(a)|.
\end{eqnarray*}
\end{proof}

\begin{proof}[Proof of Lemma \ref{lemma_amenable_subset_distinct_vertices}]
Let $\varepsilon>0$ and let $c\in C_k^{\lf}(X,Y)$ be a $Z$-non-degenerated cycle. By Proposition \ref{prop_antisymmetrisation} we can assume that $c$ is antisymmetric, in particular it is $Z$-antisymmetric. We apply Lemma \ref{lemma_Z-admissible} and obtain a $Z$-antisymmetric, $Z$-locally $K_*(X)$-admissible finite approximation $c'$ of $c$ such that $|c'|^{ne(Z)}_1\leq |c|^{ne(Z)}_1$. Consider a function
\[
c'|_{K_k(X)\cap e(Z)}:K_k(X)\cap e(Z)\rightarrow\mathbb{R}.
\]
It is finitely supported by the compactness of $Z$. Let $\mu\in\ell^1(\Gamma_k(X,Z))$ be a probability measure given by Corollary \ref{corollary_Gromov_diffusion} for $c'|_{K_k(X)\cap e(Z)}$ and the action of $\Gamma_k(X,Z)$ on $K_k(X)\cap e(Z)$. The cycle $\mu*c'$ is obviously a finite approximation of $c'$ (hence $c$). We claim that
\[
|\mu*c'|_1\leq |c'|_1^{ne(Z)}+\varepsilon \leq |c|_1^{ne(Z)} +\varepsilon.
\]
Note that because the action of $\Gamma_k(X,Z)$ preserves $K_k(X)\cap e(Z)$ we have $|\mu*c'|^{ne(Z)}_1\leq  |c'|_1^{ne(Z)}$, hence it suffices to show that
\[
|\mu*c'|_1^{e(Z)}\leq \varepsilon.
\]
Let $A_1,...,A_N\subset K_k(X)$ be the orbits of $\supp(c')\cap e(Z)$ by the $\Gamma_k(X,Z)$-action. By Corollary \ref{corollary_Gromov_diffusion} we have
\[
|\mu*c'|_1^{e(Z}\leq \varepsilon + \sum_{i=1}^N |\sum_{\sigma\in A_i} c'(\sigma)|.
\]
We will show that for every $A_i\subset K_k(X)$ for $i=1,...,N$ we have $|\sum_{\sigma\in A_i} c'(\sigma)| = 0$. Let $\sigma\in\supp(c')\cap e(Z)$ and let $s\in \Upsilon^Z_{\sigma}$ be a non-trivial transposition interchanging some vertices of $\sigma$ which are joined by an edge contained in $Z$. By the second part of Lemma \ref{lemma_properties_Gamma} there is an element $g\in\Gamma_k(X,Z)$ such that $g\cdot\sigma = s\cdot\sigma$. Therefore we have
\[
c'(g\cdot\sigma) = c'(s\cdot\sigma) = -c'(\sigma)
\]
by the $Z$-antisymmetry of $c'$. In particular, we can divide all the simplices in $\supp(c)\cap A_i$ into two groups, which coefficients exactly cancel out. It follows that $|\sum_{\sigma\in A_i} c'(\sigma)| = 0$. Because $\varepsilon$ was arbitrary, the lemma follows.
\end{proof}

\subsection{Local barycentric subdivision}\label{subsecton_local_barycentric}
We concentrate now on the proof of Lemma \ref{lemma_local_barycentric}. In general there is no obvious way to modify a chain without increasing $\ell^1$-norm such that its modified simplices have distinct vertices. However, there is an easy method allowing to do so if we drop the constraint on $\ell^1$-norm, by using barycentric subdivision operator $S$ (with some minor modifications). On the other hand, by Lemma \ref{lemma_amenable_subset_distinct_vertices} the inflation of $\ell^1$ norm is not a problem for us as long as we add only simplices with some edges in $Z$. This can be achieved using a \emph{local barycentric subdivision}, described further in this section.

Throughout this section we assume that $(X,Y)$ is a pair of topological spaces and $Z\subset X$ is compact.

\begin{defi}
We say that a simplex $\sigma\in C(\Delta^k, X)$ for $k\in\mathbb{N}$ is \emph{$Z$-barycentrically non-degenerated} if
\begin{itemize}
\item for every face $\sigma'$ of $\sigma$ with vertices in $Z$, the barycenter of $\sigma'$ is distinct from the barycenters of non-trivial faces of $\sigma'$;
\item for every face $\sigma'$ of $\sigma$ with vertices in $Z$, the barycenter of $\sigma'$ is also in $Z$ and at least one edge in $S(\sigma')$ joining this barycenter with some vertex of $\sigma'$ is contained in $Z$;
\item for every edge $e\in \sigma^{(1)}\cup (S\sigma)^{(1)}$ with endpoints in $Z$, if $e$ is homotopic (relative to its endpoints) to an interval in $Z$ then $e$ is contained in $Z$.
\end{itemize}

We say that a chain $c\in C_*^{\lf}(X,Y)$ is \emph{$Z$-barycentrically non-degenerated} if every $\sigma\in\supp(c)$ is $Z$-barycentrically non-degenerated.
\end{defi}

Note that by Lemma \ref{lemma_local_chain_homotopy} we can assume that a given chain $c\in C^{\lf}_*(X,Y)$ is $Z$-barycentrically non-degenerated.

\begin{prop}\label{prop_local_barycentric_subdivision}
There exists a chain operator $S_Z:C^{\lf}_*(X,Y)\rightarrow S_Z:C_*^{\lf}(X,Y)$ such that
\begin{enumerate}
\item $S_Z$ is chain homotopic to the identity;
\item $S_Z=Id$ when restricted to chains consisting of simplices without vertices in $Z$;
\item $S_Z = S$ when restricted to chains consisting of simplices with all vertices in $Z$;
\item for any $Z$-barycentrically non-degenerated chain $c\in C^{\lf}_*(X)$ one has $|S_Z(c)|_1^{ne(Z)}\leq |c|_1^{ne(Z)}$
\end{enumerate}
We call the above operator a \emph{local barycentric subdivision operator}.
\end{prop}

In particular, it follows from the compactness of $Z$ that for every cycle $c\in C^{\lf}_*(X,Y)$ the cycle $S_Z(c)$ is a finite approximation of $c$.

Having the above proposition, we are ready to prove Lemma \ref{lemma_local_barycentric}.

\begin{proof}[Proof of Lemma \ref{lemma_local_barycentric}]
Given a cycle $c\in C_*^{\lf}$, we can assume it is $Z$-barycentrically non-degenerated. Then the cycle $S_Z(c)$ is $Z$-non-degenerated and by Proposition \ref{prop_local_barycentric_subdivision} it is a finite approximation of $c$ such that $|S_Z(c)|^{ne(Z)}_1 \leq |c|^{ne(Z)}_1$.
\end{proof}

The rest of this section is devoted to the proof of Proposition \ref{prop_local_barycentric_subdivision}. To prove that the local barycentric subdivision is homotopic to the identity for finite chains it would suffice to use acyclic model theorems. However, for locally finite chains we will need the following lemma.

\begin{lemma}\label{lemma_chain_homotopy_appendix}
For a pair of topological spaces $(X,Z)$ let $T^{(X,Z)}_*: C_*(X)\rightarrow C_*(X)$ be a chain operator such that $T^{(X,Z)}_0=Id_0$ and $T^{(X,Z)}_*$ is functorial in the sense that for any continuous map $f:X\rightarrow X'$ we have $T^{(X',f(Z))}f_* = f_*T^{(X,Z)}$. Then $T^{(X,Z)}$ is chain homotopic to the identity. Moreover, the same is true for the simplex-wise extension of $T^{(X,Z)}$ to locally finite chains.
\end{lemma}

\begin{proof}
By induction on $k\in\mathbb{N}$ we will construct a functorial chain homotopy joining $T^{(X,Z)}$ and $Id$, i.e. an operator $P^{(X,Z)}_k:C_k(X)\rightarrow C_{k+1}(X)$ such that
\begin{itemize}
\item $\partial P^{(\Delta^n, Z)}_* + P^{(\Delta^n, Z)}_{*-1}\partial = T_*^{(\Delta^n, Z)} - Id_*$ for $Z\subset\Delta^n$;
\item if $\sigma\in C(\Delta^k,X)$ and $Z\subset X$ then
\[
P^{(X, Z)}_k(\sigma) = \sigma_* P_k^{(\Delta^k,\sigma^{-1}(Z))}(\Delta^k).
\]
\end{itemize}
For $k=0$ we can put $P^{(X,Z)}_0 = 0$. Assume that the operator $P^{(X,Z)}_k$ with the above properties is constructed for $k<n$. To construct it for $k=n$ it suffices to construct $P^{(\Delta^n,Z)}_n(\Delta^n)$ for $Z\subset\Delta^n$. Indeed, by functoriality we would then have
\[
P_n^{(X,Z)}(\sigma) = \sigma_*(P_n^{(\Delta^n,\sigma^{-1}(Z))}(\Delta^n))
\]
for arbitrary simplex $\sigma\in C(\Delta^n, X)$. Moreover,
\begin{eqnarray*}
P^{(X,Z)}_{n-1}\partial\sigma & = & \sum_{j=0}^n(-1)^j P^{(X,Z)}_{n-1}\partial_j\sigma \\
& = & \sum_{j=0}^n(-1)^j (\partial_j\sigma)_*(P^{(\Delta^{n-1},(\partial_j\sigma)^{-1}(Z))}(\Delta^{n-1})) \\
& = & \sum_{j=0}^n(-1)^j (\sigma\circ\delta^j)_*(P^{(\Delta^{n-1},(\sigma\circ\delta^j)^{-1}(Z))}(\Delta^{n-1})) \\
& = & \sum_{j=0}^n(-1)^j \sigma_*(P^{(\Delta^n,\sigma^{-1}(Z))}(\delta^j\Delta^{n-1})) \\
& = & \sigma_*(P^{(\Delta^n,\sigma^{-1}(Z))}(\partial\Delta^n)).
\end{eqnarray*}
It would follow that
\begin{eqnarray*}
(\partial P^{(X,Z)}_n + P^{(X,Z)}_{n-1}\partial)(\sum_i a_i \sigma_i) & = & \sum_i a_i (\sigma_i)_*(\partial P_n^{(\Delta^n,\sigma^{-1}(Z))}(\Delta^n) +P_{n-1}^{(\Delta^n,\sigma^{-1}(Z))}(\partial\Delta^n)) \\
& = &  \sum_i a_i (\sigma_i)_*(T^{(\Delta^n,\sigma^{-1}(Z))}_n(\Delta^n)-\Delta^n) \\
& = & (T^{(X,Z)}_n - Id_n)(\sum_i a_i \sigma_i).
\end{eqnarray*}

We come back to the definition of $P^{(\Delta^n,Z)}_n(\Delta^n)$. For $n\in\mathbb{N}$ denote by $c_n:C_*(\Delta^n)\rightarrow C_{*+1}(\Delta^n)$ an operator that cones the singular simplices with the barycenter $b_n$ of $\Delta^n$, i.e. for $\sigma\in C(\Delta^k, \Delta^n)$ we have
\[
c_n(\sigma)(t_0,...,t_{k+1}) = t_0\cdot b_n + (1-t_0)\cdot \sigma(\frac{t_1}{1-t_0},...,\frac{t_{k+1}}{1-t_0}).
\]
Note that we have an equality
\[
\partial c_n(\sigma) = \sigma - c_n(\partial\sigma).
\]
We define
\[
P^{(\Delta^n,Z)}_n(\Delta^n):= c_n(T^{(\Delta^n,Z)}_n(\Delta^n) - \Delta^n -P^{(\Delta^n,Z)}_{n-1}(\partial\Delta^n))
\]
and we check that
\begin{eqnarray*}
(\partial P^{(\Delta^n,Z)}_n + P^{(\Delta^n,Z)}_{n-1}\partial)(\Delta^n) & = & T^{(\Delta^n,Z)}_n(\Delta^n) - \Delta^n - c_n(\partial T_n^{(\Delta^n,Z)}(\Delta^n)) +c_n(\partial\Delta^n) +c_n(\partial P^{(\Delta^n,Z)}_{n-1}(\partial\Delta^n)) \\
& = & T^{(\Delta^n,Z)}_n(\Delta^n) - \Delta^n - c_n(T^{(\Delta^n,Z)}_{n-1}(\partial\Delta^n)) +c_n(\partial\Delta^n) \\
& & + c_n(T^{(\Delta^n,Z)}_{n-1}(\partial\Delta^n) - \partial\Delta^n -  P^{(\Delta^n,Z)}_{n-2}(\partial\partial\Delta^n)) \\
& = & T^{(\Delta^n,Z)}_n(\Delta^n) -\Delta^n,
\end{eqnarray*}
where in the second equality we used the fact that $T^{(\Delta^n,Z)}_*$ is a chain operator and the inductive hypothesis that
\[
\partial P^{(\Delta^n,Z)}_k = T^{(\Delta^n,Z)}_k - Id_k - P^{(\Delta^n,Z)}_{k-1}\partial
\]
for $k<n$.

Finally, we observe that because of the functorial nature of $T^{(X,Z)}_*$, for every $k\in\mathbb{N}$ and every simplex $\sigma\in C(\Delta^k, X)$, all the simplices in $\supp(T^{(X,Z)}_k(\sigma))$ and $\supp(P^{(X,Z)}_k(\sigma))$  are contained in $\sigma$, hence the simplex-wise extension of $T^{(X,Y)}_*$ to locally finite chains is well defined and homotopic to the identity.
\end{proof}

Another ingredient of the proof of Proposition \ref{prop_local_barycentric_subdivision} is the following lemma, which will help us to bound the norm $|S_Z(c)|_1^{ne(Z)}$.

\begin{lemma}\label{lemma_counting_edges_main}
Let $\sigma\in C(\Delta^k, X)$ be a $Z$-non-degenerated simplex with all vertices in $Z$. Then
\begin{enumerate}
\item if $\sigma\in ne(Z)$ then there exists exactly one simplex $\sigma'\in\supp(S\sigma)$ with no edges in $Z$;\label{lemma_counting_noedge}
\item if $\sigma\in e(Z)$ then $\supp(S\sigma)\subset e(Z)$.\label{lemma_counting_edge}
\end{enumerate}
\end{lemma}

\begin{proof}
\leavevmode
\begin{enumerate}
\item Assume $\sigma\in ne(Z)$. Colour the vertices of $S\Delta^k$ such that two vertices have the same colour if and only if they are connected by a path in $(S\Delta^k)^{(1)}$ which is mapped by $\sigma$ to $Z$. By the second condition of $Z$-barycentrically non-degenerateness every vertex $v\in (S\Delta^k)^{(0)}$ has the same colour as some vertex of the minimal face of $\Delta^k$ containing $v$. Moreover, from the third condition it follows that if two vertices $v,w$ of $S\Delta^k$ have the same colour, the interval in $\Delta^k$ joining them is mapped by $\sigma$ to $Z$.

We will show by induction that there exists exactly one simplex $\Delta'\in\supp S(\Delta^k)$ with vertices with distinct colours, which will end the proof. It is clear for $k=1$, so let $k>1$. Assume by induction hypothesis that for every face $\Delta'$ of $\Delta^k$ of codimension one there exists exactly one simplex $\Delta''\in\supp(S\Delta')$ with distinctly coloured vertices. Moreover, these colours are the same as the colours of the vertices of $\Delta'$. Note that every $\bar{\Delta}\in\supp S\Delta^k$ is a cone with the barycenter of $\Delta^k$ as a vertex and a simplex $\bar{\Delta}'\in\supp S\partial\Delta^k$ as a base. It follows that anyhow we colour the barycenter, there will be only one simplex in $\supp S\Delta^k$ that has distinctly coloured vertices.

\item Assume now that $\sigma\in e(Z)$. Repeat the argument above, but start with vertices of $\Delta^k$ coloured in the same colour whenever the corresponding edge of $\sigma$ joining them is contained in $Z$. Then, by the assumption, there are at most $k$ available colours, and all simplices in $\supp S\Delta^k$ have $k+1$ vertices, hence the lemma follows also in this case.
\end{enumerate}
\end{proof}

\begin{proof}[Proof of Proposition \ref{prop_local_barycentric_subdivision}]
In this proof we denote by $F(\Delta^k)$ the set of faces of $\Delta^k$ and by $v_0,...,v_k$ the vertices of $\Delta^k$.

Let $\sigma\in C(\Delta^k, X)$ be a singular simplex and let $\mathbb{V}_{\sigma}\subset\{v_0,..., v_k\}$ be the vertices of $\Delta^k$ which are mapped by $\sigma$ to $Z$. We define $S_Z:C(\Delta^k,X)\rightarrow C_*(X)$ to be the identity for $k=0$ and for $k>0$
\[
S_Z(\sigma) := \begin{cases} \sigma_*(b^{\Delta^k}_{\sigma^{-1}(Z)}S_{\sigma^{-1}(Z)}\partial_{\sigma^{-1}(Z)}(\Delta^k)) & \text{ if }\mathbb{V}_{\sigma}\neq \emptyset; \\
\sigma & \text{ if } \mathbb{V}_{\sigma}= \emptyset,
 \end{cases}
\]
where for $W\subset \Delta^k$ and $\Delta'\in F(\Delta^k)$
\begin{itemize}
\item $\partial_W:C(\Delta^*,\Delta^k)\rightarrow C_{*-1}(\Delta^k)$ is defined to be
\[
\partial_W\tau = \sum_{i\::\: \tau(v_i)\in W}(-1)^i\partial_i\tau
\]
for $\tau\in C(\Delta^*,\Delta^k)$;
\item $b^{\Delta'}_W:C(\Delta^*,\Delta^k)\rightarrow C_{*+1}(\Delta^k)$ is defined to be
\[
b^{\Delta'}_W(\tau)(t_0,...,t_{l+1}) = t_0\cdot b^{\Delta'}(W) +(1-t)\cdot\tau(\frac{t_1}{1-t_0},...,\frac{t_{l+1}}{1-t_0})
\]
for $\tau\in C(\Delta^l,\Delta^k)$ and $(t_0,...,t_{l+1})\in \Delta^{l+1}$, where $b^{\Delta'}(W)$ is the barycenter of the maximal face in $F(\Delta')$ with vertices in $W$ (i.e. the face spanned by $(\Delta')^{(0)}\cap W$). Note that
\[
\partial b^{\Delta'}_W(\tau) = \tau - b^{\Delta'}_W\partial\tau.
\]
Moreover, if $\Delta''\in F(\Delta')$ and $(\Delta')^{(0)}\cap W\subset \Delta''$, then $b_W^{\Delta'} = b_W^{\Delta''}$.
\end{itemize}
See Figure \ref{figure_local_barycentric}. It is also good to compare this definition with a definition of the standard barycentric subdivision \cite{HAT}.

\begin{figure}[h]
\centering
\includegraphics[width=0.5\textwidth]{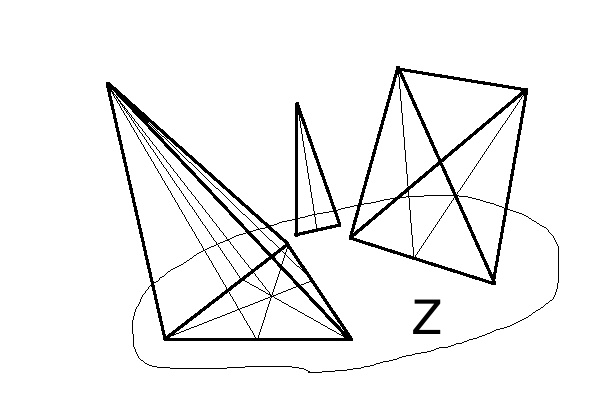}
\caption{Locally barycentricaly subdivided simplices.}
\label{figure_local_barycentric}
\end{figure}

Now we need to check all the properties in the thesis of the proposition for $S_Z$. First of all, we need to check that $S_Z$ is a chain operator, but before doing that, let us prove two technical claims. Let $k\in\mathbb{N}$, $W\subset\Delta^k$ and $\tau\in C(\Delta^*,\Delta^k)$.

\emph{Claim 1}: $\partial_W\partial_W = 0$.
\begin{eqnarray*}
\partial_W\partial_W(\tau) & = & \sum_{i\::\: \tau(v_i)\in W}(-1)^i\partial_W\partial_i\tau \\
& = & \sum_{\{i:\: \tau(v_i)\in W\}}\sum_{\{j<i:\: \tau(v_j)\in W\}}(-1)^{i+j}\partial_j\partial_i\tau + \sum_{\{i:\: \tau(v_i)\in W\}}\sum_{\{j> i:\: \tau(v_j)\in W\}}(-1)^{i+j-1}\partial_{j-1}\partial_i\tau \\
& = & \sum_{\{i>j:\: \{\tau(v_i),\tau(v_j)\}\subset W\}}(-1)^{i+j}\partial_j\partial_i\tau + \sum_{\{i<j:\: \{\tau(v_i),\tau(v_j)\}\subset W\}} (-1)^{i+j-1}\partial_i\partial_j\tau \\
& = & \sum_{\{i>j:\: \{\tau(v_i),\tau(v_j)\}\subset W\}}(-1)^{i+j}\partial_j\partial_i\tau + (-1)^{i+j-1}\partial_j\partial_i\tau  \\
& = & 0.
\end{eqnarray*}

\emph{Claim 2}: $\partial\partial_W = -\partial_W\partial$.
\begin{eqnarray*}
\partial\partial_W(\tau) = & = & (\partial-\partial_W+\partial_W)\partial_W(\tau) \\
& = & (\partial-\partial_W)\partial_W(\tau) \\
& = & \sum_{i\::\: \tau(v_i)\in W}(-1)^i(\partial-\partial_W)\partial_i\tau \\
& = & \sum_{\{i:\: \tau(v_i)\in W\}}\sum_{\{j<i:\: \tau(v_j)\notin W\}}(-1)^{i+j}\partial_j\partial_i\tau + \sum_{\{i:\: \tau(v_i)\in W\}}\sum_{\{j>i:\: \tau(v_j)\notin W\}}(-1)^{i+j-1}\partial_{j-1}\partial_i\tau \\
& = & \sum_{\{i>j:\: \tau(v_i)\in W,\, \tau(v_j)\notin W \}}(-1)^{i+j}\partial_j\partial_i\tau + \sum_{\{i<j:\: \tau(v_i)\in W,\, \tau(v_j)\notin W\}} (-1)^{i+j-1}\partial_{j-1}\partial_i\tau \\
& = & \sum_{\{i>j:\: \tau(v_i)\in W,\, \tau(v_j)\notin W \}}(-1)^{i+j}\partial_{i-1}\partial_j\tau + \sum_{\{i<j:\: \tau(v_i)\in W,\, \tau(v_j)\notin W\}} (-1)^{i+j-1}\partial_i\partial_j\tau \\
& = & -\sum_{j\::\: \tau(v_j)\notin W}(-1)^j(\partial_W)\partial_j\tau \\
& = & -\partial_W(\partial-\partial_W)(\tau) \\
& = & -\partial_W\partial(\tau).
\end{eqnarray*}

Having the above claims proved, we are ready to prove that $S_Z$ is a chain map. If $\sigma$ has no vertices in $Z$, then $S_Z(\sigma)=\sigma$, in particular $\partial S_Z(\sigma) = S_Z(\partial\sigma)$. Otherwise, we have by induction on $\dim \sigma$
\begin{eqnarray*}
\partial S_Z(\sigma) & = & \sigma_*\partial(b^{\Delta^k}_{\sigma^{-1}(Z)}S_{\sigma^{-1}(Z)}\partial_{\sigma^{-1}(Z)}(\Delta^k)) \\
& = & \sigma_*(S_{\sigma^{-1}(Z)}\partial_{\sigma^{-1}(Z)} - b^{\Delta^k}_{\sigma^{-1}(Z)}\partial S_{\sigma^{-1}(Z)}\partial_{\sigma^{-1}(Z)})(\Delta^k) \\
& = & \sigma_*(S_{\sigma^{-1}(Z)}\partial_{\sigma^{-1}(Z)} - b^{\Delta^k}_{\sigma^{-1}(Z)} S_{\sigma^{-1}(Z)}\partial\partial_{\sigma^{-1}(Z)})(\Delta^k) \\
& = & \sigma_*(S_{\sigma^{-1}(Z)}\partial_{\sigma^{-1}(Z)} + b^{\Delta^k}_{\sigma^{-1}(Z)} S_{\sigma^{-1}(Z)}\partial_{\sigma^{-1}(Z)}\partial)(\Delta^k) \\
& = & \sigma_*(S_{\sigma^{-1}(Z)}\partial_{\sigma^{-1}(Z)} + b^{\Delta^k}_{\sigma^{-1}(Z)} S_{\sigma^{-1}(Z)}\partial_{\sigma^{-1}(Z)}(\partial-\partial_{\sigma^{-1}(Z)}))(\Delta^k) \\
& = & \sigma_*(S_{\sigma^{-1}(Z)}\partial_{\sigma^{-1}(Z)})(\Delta^k) + \sum_{\{i\::\:v_i\notin \sigma^{-1}(Z)\}}(-1)^i \sigma_*(b^{\Delta^k}_{\sigma^{-1}(Z)} S_{\sigma^{-1}(Z)}\partial_{\sigma^{-1}(Z)})(\partial_i\Delta^k) \\
& = & \sigma_*(S_{\sigma^{-1}(Z)}\partial_{\sigma^{-1}(Z)})(\Delta^k) + \sum_{\{i\::\:v_i\notin \sigma^{-1}(Z)\}}(-1)^i \sigma_*(b^{\partial_i\Delta^k}_{\sigma^{-1}(Z)} S_{\sigma^{-1}(Z)}\partial_{\sigma^{-1}(Z)})(\partial_i\Delta^k) \\
& = & \sigma_*(S_{\sigma^{-1}(Z)}\partial_{\sigma^{-1}(Z)})(\Delta^k) + \sum_{\{i\::\:v_i\notin \sigma^{-1}(Z)\}}(-1)^i \sigma_*(S_{\sigma^{-1}(Z)}(\partial_i\Delta^k)) \\
& = & \sigma_*(S_{\sigma^{-1}(Z)}\partial_{\sigma^{-1}(Z)})(\Delta^k) + \sigma_*(S_{\sigma^{-1}(Z)}(\partial - \partial_{\sigma^{-1}(Z)})(\Delta^k)) \\
& = & \sigma_*(S_{\sigma^{-1}(Z)}\partial)(\Delta^k) \\
& = & S_Z\partial(\sigma).
\end{eqnarray*}

Knowing that $S_Z$ is a chain operator, we can now prove that it has the desired properties.

\begin{enumerate}
\item It follows easily from Lemma \ref{lemma_chain_homotopy_appendix}.
\item It is obvious from the definition of $S_Z$.
\item It is easy to prove by induction that if $\sigma^{(0)}\subset Z$ then $S_Z\sigma = \sigma_*(b^{\Delta^k}_{\Delta^k}S\partial)(\Delta^k) = S\sigma$ (compare with the definition of barycentric subdivision from e.g. \cite{HAT}).
\item Let $W\subset\Delta^k$ be such that $\Delta^k$ is $W$-barycentrically non-degenerated. For $\tau\in C(\Delta^*, \Delta^k)$ we denote by $\tau_W$ the maximal face of $\tau$ with the vertices in $W$. The key observation is that if $\Delta^k_W\neq\emptyset$ then there is a bijection
\begin{eqnarray*}
\supp(S_W(\Delta^k)) & \rightarrow & \supp(S(\Delta^k_W)), \\
\tau & \mapsto & \tau_W.
\end{eqnarray*}
It can be proved by induction on $k$. For $k=0$ the observation is obvious, so assume $k>0$. Note that the elements of $\supp(\partial_W\Delta^k)$ correspond to codimension-$1$ faces of $\Delta^k_W$ by the bijection $\Delta''\mapsto \Delta''\cap \Delta^k_W = \Delta''_W$. Using the induction hypothesis it follows that there is a bijection
\begin{eqnarray*}
\supp(S_W\partial_W(\Delta^k)) & \rightarrow & \supp(S\partial \Delta^k_W); \\
\Delta'' & \mapsto & \Delta''\cap \Delta^k_W.
\end{eqnarray*}
Because $b^{\Delta^k}_W$ is an operator of taking cone with a vertex in $W$, we have a bijection
\begin{eqnarray*}
\supp(S_W(\Delta^k)) = \supp(b^{\Delta^k}_W S_W\partial_W(\Delta^k)) & \rightarrow & \supp(b^{\Delta^k}_W S\partial(\Delta^k_W)) = \supp(S(\Delta^k_W)) \\
\tau = b^{\Delta^k}_W\tau' & \mapsto & b^{\Delta^k}_W(\tau'_W) = (b^{\Delta^k}_W\tau')_W = \tau_W.
\end{eqnarray*}
This finishes the proof of the observation.

By the above observation, we conclude that for $\sigma\in\supp(c)$ such that $\sigma_Z\neq\emptyset$ the simplices in $\supp(S_Z(\sigma))$ are constructed from the simplices in $\supp(S(\sigma_Z))$ by taking cones multiple times, hence by Lemma \ref{lemma_counting_edges_main} $\supp(S_Z(\sigma))\subset e(Z)$ if $\sigma\in e(Z)$ and there is exactly one simplex in $\supp(S_Z(\sigma))\cap ne(Z)$ if $\sigma\in ne(Z)$. It is also obvious that $S_Z=Id$ for simplices with no vertices in $Z$, hence 
\[
|S_Z(c)|^{ne(Z)}_1 \leq |c|^{ne(Z)}_1.
\]
\end{enumerate}
\end{proof}

\section{Superadditivity of the locally finite simplicial volume I}\label{section_amenable_trick}
Having proved the subadditivity part of Theorem \ref{Theorem_main}, we turn our attention to superadditivity part. Our goal in this section is to prove the following proposition, which implies the superadditivity part of Theorem \ref{Theorem_main} under the assumption of asphericity of a boundary piece that we use to glue manifolds.

\begin{prop}\label{prop_chain_decomposition}
Let $(X_1,Y_1)$ and $(X_2,Y_2)$ be two pairs of $CW$-complexes and let $Z_i\subset Y_i$ for $i=1,2$ be two compact path-connected aspherical components such that there is a homeomorphism $f:Z_1\rightarrow Z_2$. Assume moreover that $\pi_1(Z_1)$ is amenable and injects both into $\pi_1(X_1)$ and $\pi_1 (X_2)$. Let $k\geq 2$ and let $c\in C_k^{\lf}(X_1\cup_f X_2, (Y_1\setminus Z_1)\cup (Y_2\setminus Z_2))$ be a cycle. Then for every $\varepsilon>0$ there exist cycles $c_1\in C^{\lf}_k(X_1, Y_1)$ and $c_2\in C^{\lf}_k(X_2, Y_2)$ which are finite approximations of restrictions of $c$ to $C_k^{\lf}(X_1\cup_f X_2, Y_1\cup X_2)$ and $C_k^{\lf}(X_1\cup_f X_2, X_1\cup Y_2)$ respectively, such that
\[
|c_1|^{ne(Z)}_1+|c_2|^{ne(Z)}_1 \leq |c|_1+\varepsilon.
\]
In particular, $\|[c_1]\|_1+\|[c_2]\|_1 \leq \|[c]\|_1$ by Proposition \ref{prop_amenable_subset}.
\end{prop}

For the rest of this section, we will write for short $X := X_1\cup_f X_2$, $Y := Y_1\cup_f Y_2$ and $Z := Z_1\cong Z_2$. The strategy is to construct 'projections' $\rho^{\varepsilon}_i: C_*^{\lf}(X,Y)\rightarrow C^{\lf}_*(X_i, Y_i)$ for $i=1,2$ and arbitrary $\varepsilon$ such that for every $c\in C_k^{\lf}(X, Y)$
\[
|\rho_1^{\varepsilon}(c)|_1^{ne(Z)} + |\rho_2^{\varepsilon}(c)|_1^{ne(Z)}\leq |c|_1 + \varepsilon.
\]
Together with Lemma \ref{prop_amenable_subset} it would yield Proposition \ref{prop_chain_decomposition}.

Before we start the construction, we will need the following simple technical lemma, which follows easily from Lemma \ref{lemma_local_chain_homotopy}.

\begin{lemma}\label{lemma_Z_compatibility}
Let $c\in C_k^{\lf}(X, Y)$ be a cycle. Then there is a finite approximation $c'$ of $c$ such that $|c'|_1\leq |c|_1$ and every $\sigma\in\supp(c')$ has the following properties:
\begin{enumerate}
\item for every edge $e$ of $\sigma$ and its lift $\wt{e}$ to the universal covering $\pi:\wt{X}\rightarrow X$ the number of connected components of $\wt{e}\cap \pi^{-1}(X\setminus Z)$ is minimal in the set of paths homotopic to $\wt{e}$ relative to its endpoints;
\item there exists $N$ (depending on $\sigma$) such that $S^{(N)}(\sigma)\in C_*^{\lf}(X_1, Y_1)\oplus C_*^{\lf}(X_2, Y_2)$.
\end{enumerate}
We say that such a simplex $\sigma$ (resp. a cycle $c'$) is \emph{$Z$-compatible}.
\end{lemma}

Note that by the compactness of $Z$ and local finiteness of a $Z$-compatible cycle $c$ there is a number $K\in\mathbb{N}$ such that $S^{(K)}(c)\in C_*^{\lf}(X_1, Y_1)\oplus C_*^{\lf}(X_2, Y_2)$.

The main lemma which we will prove in this section is the following. Recall that by $n(Y)\subset C(\Delta^*, X)$ we denote the set of singular simplices not contained in $Y$.

\begin{lemma}\label{lemma_boundary_trick}
Let $c\in C_*^{\lf}(X,Y)$ be a $Z$-compatible cycle for $*\geq 2$ which is an image of a cycle from $C_*^{\lf}(X,Y\setminus Z)$ and let $K\in\mathbb{N}$ be a number such that $S^{(K)}(c)\in C_*^{\lf}(X_1,Y_1)\oplus C_*^{\lf}(X_2,Y_2)$. For every $\varepsilon>0$ there exists a finite modification $c'\in C_*^{\lf}(X_1,Y_1)\oplus C_*^{\lf}(X_2,Y_2)$ of $c$ such that
\begin{itemize}
\item $S^{(K)}(c') = S^{(K)}(c)$ in $C_*^{\lf}(X_1,Y_1)\oplus C_*^{\lf}(X_2,Y_2)$;
\item $|c'|^{ne(Y)}_1\leq |c|^{ne(Y)}_1$;
\item $\supp((\partial c)|_{n(Y)})$ is finite and $|(\partial c)|_{n(Y)}|_1\leq \varepsilon$;
\end{itemize}
\end{lemma}

\begin{proof}[Proof of Proposition \ref{prop_chain_decomposition}]
By Lemma \ref{lemma_Z_compatibility} we can assume $c$ is $Z$-compatible. Let $\varepsilon'>0$ and let $c'$ be a chain given by Lemma \ref{lemma_boundary_trick} for constant $K$ and $\varepsilon$. Let also $T$ be a chain homotopy between $S^{(K)}$ and $Id$. We claim that the chain
\[
c'':=c'+T((\partial c')|_{n(Y)})
\]
which lies in $C_*^{\lf}(X_1,Y_1)\oplus C_*^{\lf}(X_2,Y_2)$, is a finite approximation of $c$. Because adding to $c''$ a term from $C_*^{\lf}(Y_1)\oplus C_*^{\lf}(Y_2)$ will not affect neither the fact that it is a cycle nor that it is homologuous to $c$, we will check these properties for $c''' = c'+T\partial c'$.
\[
\partial c''' = \partial c' +\partial T\partial c' = \partial c' + S^{(K)}\partial c' - \partial c' -T\partial\partial c' = \partial S^{(K)} c' = \partial S^{(K)} c
\]
The last term is in $C_*^{\lf}(Y_1)\oplus C_*^{\lf}(Y_2)$, hence $c'''$ (and $c''$) is a cycle. We used the fact that $S^{(K)}$ is a chain operator and the properties of $c'$.

Now we check that $c'''$ is homologous to $c$:
\[
[c'''] = [S^{(K)}(c''')] = [S^{(K)}c' + S^{(K)}T\partial c'] = [S^{(K)}c' + S^{(2K)}c' - S^{(K)}c' - S^{(K)}\partial Tc'] = [S^{(2K)}c'] = [S^{(2K)}c] = [c].
\]

Finally, let us estimate the norm of $c''$.
\[
|c''|^{ne(Z)}_1 \leq |c'|^{ne(Z)}_1 + |T\partial (c'|_{n(Y)})|_1 \leq |c|_1 + \varepsilon' \|T\|.
\]
Because $\varepsilon'$ is arbitrary and $\|T\|$ depends only on $c$, we can set $\varepsilon' = \frac{\varepsilon}{\|T\|}$ and obtain
\[
|c''|^{ne(Z)}_1 \leq |c|_1 + \varepsilon.
\]
The components of $c''$ provide desired cycles.
\end{proof}

Now we turn to the proof of Lemma \ref{lemma_boundary_trick}, but first let us make a few remarks about the structure of $\wt{X}$, the universal covering of $X$. Because the fundamental groups $\pi_1(X_i)$ for $i=1,2$ and $\pi_1(Z)$ inject into $\pi_1(X)$, $\wt{X}$ is built from the copies of universal coverings of $X_i$ for $i=1,2$ glued along copies of universal coverings of $Z$. Let $\hat{X_i}$ for $i=1,2$ be distinguished copies of $\wt{X_i}$ inside $\wt{X}$, which intersect along a copy of $\wt{Z}$, which we denote by $\hat{Z}$. For $i=1,2$ we define also
\[
\mathcal{X}_i = \{\gamma \hat{X_i}\::\: \gamma\in\pi_1(X)\}.
\]
Because $\hat{Z}$ is contractible, there exists a retraction $r:\wt{X}\rightarrow \hat{Z}$. It can be constructed simply by sending $\wt{X}$ to a point and then by extending a homotopy between a constant map and the identity on $\hat{Z}$ to $\wt{X}$. Moreover, the group $\pi_1(Z)$ acts on the set of such retractions by
\[
(g\cdot r)(x) = g\cdot r(g^{-1}\cdot x),
\]
where $x\in \wt{X}$, $g\in\pi_1(Z)$ and $r:\wt{X}\rightarrow \hat{Z}$ is a retraction. We check that it is indeed an action.
\[
(g\cdot (h\cdot r))(x) = g\cdot (h\cdot r)(g^{-1}\cdot x) = gh\cdot r(h^{-1}g^{-1}\cdot x) = ((gh)\cdot r) (x).
\]
Having chosen one particular retraction $r_{\hat{Z}}$, we define retractions $\wt{X}\rightarrow \hat{X_i}$ for $i=1,2$ as follows. Choose some representatives of cosets of $\pi_1(X_i)/\pi_1(Z)$ and denote them by $(h^i_j)_{j\in\pi_1(X_i)/\pi_1(Z)}$. Let $\bar{X_i}$ be the connected component of $\wt{X}\setminus\hat{Z}$ that does not contain $\hat{X_i}$. Then the connected components of $\wt{X}\setminus\hat{X_i}$ are exactly $h^i_j\bar{X_i}$ for $j\in\pi_1(X_i)/\pi_1(Z)$ (see figure \ref{figure_universal_covering_tree}). Define
\[
r^{(h^i_j)_{j\in\pi_1(X_i)/\pi_1(Z)}}_{\hat{X_i}}(x) = \begin{cases} 
x & \text{ if } x\in \hat{X_i}; \\
h^i_j\cdot r_{\hat{Z}}((h^i_j)^{-1} x) & \text{ if } x\in h^i_j\bar{X_i}. 
\end{cases}
\]
We will denote by $\mathcal{R}_i$ the set of such retractions, i.e.
\[
\mathcal{R}_i := \{r^{(h^i_j)_{j\in\pi_1(X_i)/\pi_1(Z)}}_{\hat{X_i}}\::\: (h^i_j)_{j\in\pi_1(X_i)/\pi_1(Z)} \text{ are representatives of cosets of }\pi_1(X_i)/\pi_1(Z) \}
\]
Note that $\prod_{j\in\pi_1(X_i)/\pi_1(Z)}\pi_1(Z)$ acts transitively on $\mathcal{R}_i$ by
\[
(g_j)_{j\in\pi_1(X_i)/\pi_1(Z)}\cdot r^{(h^i_j)_{j\in\pi_1(X_i)/\pi_1(Z)}}_{\hat{X_i}} = r^{(h^i_j(g_j)^{-1})_{j\in\pi_1(X_i)/\pi_1(Z)}}_{\hat{X_i}}.
\]

\begin{figure}[h]
\centering
\includegraphics[width=0.5\textwidth]{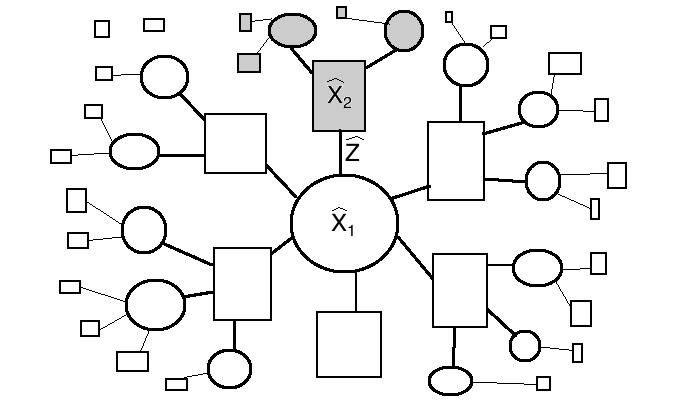}
\caption{Sketch picture of $\wt{X}$. Circles correspond to the copies of $\wt{X_1}$, squares to the copies of $\wt{X}_2$ and intervals joining them to the copies of $\wt{Z}$. Gray region represents $\bar{X}_1$, while white one-$\bar{X}_2$.}
\label{figure_universal_covering_tree}
\end{figure}

Let $(H^i_j)_{j\in\pi_1(X_i)\setminus\pi_1(X)}$ be some set of representatives for $\pi_1(X_i)\setminus\pi_1(X)$ (we will not change it during our considerations). In particular, copies of $\wt{X_i}$ in $\wt{X}$ are exactly $(H^i_j)^{-1}\hat{X_i}$. Define
\[
R:\mathcal{R}_1\times\mathcal{R}_2\times C_*(\wt{X})\rightarrow C_*(\wt{X}, \pi^{-1}(Z))
\]
simplex-wise as
\[
R(r_1,r_2,\sigma) = \sum_{i=1}^2\sum_{j\in\pi_1(X_i)\setminus\pi_1(X)} r_i\circ H^i_j\circ\sigma
\]
Note that because the image of the above map is in $C_*(\wt{X}, \pi^{-1}(Z))$, the sum on the right-hand side can be treated as finite because $r_i\circ H^i_j\circ\sigma\subset \pi^{-1}(Z)$ for almost all $j\in\pi_1(X_i)\setminus\pi_1(X)$. It is also clear that it is a chain map (for finite chains!).

\begin{lemma}\label{lemma_property_of_R}
For any $r_i\in\mathcal{R}_i$ for $i=1,2$, if $\sigma$ is $\pi^{-1}(Z)$-compatible and of dimension $k\geq 2$, at most one summand in $R(r_1,r_2,\sigma)$ has no edges in $\pi^{-1}(Z)$.
\end{lemma}

\begin{proof}
Consider a simplicial tree $\mathcal{T}$ with vertices indexed by copies of $\wt{X_i}$ for $i=1,2$ such that two vertices are joined by an edge if the corresponding copies of $\wt{X_1}$ and $\wt{X_2}$ are intersecting in $\wt{X}$ along some copy of $\wt{Z}$. Then for a simplex $\sigma\in C(\Delta^k, \wt{X})$ there is a corresponding simplex $\sigma'\in C(\Delta^k, \mathcal{T})$ with geodesic edges (by $\pi^{-1}(Z)$-compatibility). Note that if an edge of $\sigma$ is mapped by $r_i\circ H^i_j$ to an edge not contained in $\pi^{-1}(Z)$ then the corresponding edge in $\sigma'$ contains the vertex $(H^i_j)^{-1}\hat{X_i}$. In other words, $r_i\circ H^i_j\circ\sigma$ has no edges in $\pi^{-1}(Z)$ if all edges of  $\sigma'$ contain $(H^i_j)^{-1}\hat{X_i}$. On the other hand, by Lemma \ref{lemma_tree} there exists at most one vertex of $\mathcal{T}$ which is contained in all the edges of $\sigma'$. It follows that there can be only one summand in $R(r_1,r_2,\sigma)$ which has no edges in $\pi^{-1}(Z)$, q.e.d.
\end{proof}

\begin{lemma}\label{lemma_tree}
Let $\mathcal{T}$ be a simplicial tree and let $\sigma\in C(\Delta^k,{\mathcal{T}})$ be a simplex with geodesic edges. If $k\geq 2$ there is at most one vertex in $\mathcal{T}$ which is contained in all edges of $\sigma$.
\end{lemma}

\begin{proof}
Assume that there are two such vertices, $v_1$ and $v_2$. Because $k\geq 2$, $\sigma$ has at least 3 edges. Let $e_1 = [v'_1, v'_2]$ and $e_2 = [v''_1, v''_2]$ be two edges of $\sigma$. Without loss of generality we assume that
\[
d_{\mathcal{T}}(v'_1, v'_2) = d_{\mathcal{T}}(v'_1,v_1)+ d_{\mathcal{T}}(v_1, v_2) + d_{\mathcal{T}}(v_2, v'_2)
\]
and 
\[
d_{\mathcal{T}}(v''_1, v''_2) = d_{\mathcal{T}}(v''_1,v_1)+ d_{\mathcal{T}}(v_1, v_2) + d_{\mathcal{T}}(v_2, v''_2).
\]
Consider an edge $e_3 = [v'_1, v''_1]$ (if $v'_1=v''_1$ and it is not an edge, then one can choose $[v'_2, v''_2]$). $v_1$ and $v_2$ are contained in $e_3$, hence
\[
d_{\mathcal{T}}(v'_1,v''_1) = d_{\mathcal{T}}(v'_1, v_2) + d_{\mathcal{T}}(v_2, v''_1) = d_{\mathcal{T}}(v'_1, v_1) + 2d_{\mathcal{T}}(v_1, v_2) + d_{\mathcal{T}}(v_1, v''_1) = d_{\mathcal{T}}(v'_1,v''_1) + 2d_{\mathcal{T}}(v_1,v_2).
\]
We conclude that $d_{\mathcal{T}}(v_1,v_2)=0$, hence $v_1=v_2$.
\end{proof}

\begin{proof}[Proof of Lemma \ref{lemma_boundary_trick}]
Let $\theta:C(\Delta^k, X)\rightarrow C(\Delta^k, \wt{X})$ be any section of the map induced by the canonical projection $\pi:\wt{X}\rightarrow X$. Take any $r_i\in\mathcal{R}_i$ for $i=1,2$ and consider a chain
\[
D(r_1,r_2, c) = \sum_{\sigma\in\supp(c)} c(\sigma) \pi\circ R(r_1,r_2,\theta(\sigma)),
\]
which lies in $C^{\lf}_*(X_1,Y_1)\oplus C^{\lf}_*(X_2,Y_2)$. $D(r_1,r_2,c)$ is not a cycle in general, because the construction depends on $\theta$. However, it is a finite modification of $c$ because taking $\theta$ and $R$ does nothing to the simplices that do not intersect $Z$ and $Z$ is compact. By Lemma \ref{lemma_property_of_R}, $|D(r_1,r_2,c)|^{ne(Z)}_1 \leq |c|_1$. Moreover, if $K$ is a constant such that $S^{(K)}(c)\in C^{\lf}_*(X_1, Y_1)\oplus C^{\lf}_*(X_2, Y_2)$ (which is given by $Z$-compatibility of $c$) then
\[
S^{(K)}(D(r_1,r_2,c)) = D(r_1,r_2,S^{(K)}(c)) = S^{(K)}(c)
\]
in $C^{\lf}_*(X_1, Y_1)\oplus C^{\lf}_*(X_2, Y_2)$. Note also that every element of $\partial D(r_1,r_2,c)|_{n(Y)}$ has non-empty intersection with $Z$, hence $\partial D(r_1,r_2,c)|_{n(Y)}$ has finite support. In particular, $D(r_1,r_2,c)$ satisfies almost all of the requirements postulated by Lemma \ref{lemma_boundary_trick} except that the norm of $\partial D(r_1,r_2,c)|_{n(Y)}$ may be large. However, we can average $D(r_1,r_2,c)$ by the action of $\bigoplus_{i=1}^2\bigoplus_{j\in\mathbb{N}}\pi_1(Z)$ and use Proposition \ref{prop_Gromov_diffusion} to minimize this norm. To describe it precisely, we need more machinery.

For $\sigma\in C(\Delta^k, X)$ and $i=1,2$ consider the sets
\[
\mathcal{U}^i(\sigma) = \{(\wt{\sigma}, r) \::\: \wt{\sigma}\in C(\Delta^{k-1}, \wt{X}) \text{ is a lift of }\sigma,\; r\in\mathcal{R}_i,\; \wt{\sigma}\cap (\hat{X_i}\setminus \pi^{-1}(Z))\neq\emptyset \} / \sim
\]
where $(\tau, r)\sim (\tau',r')$ if there exists $g\in \pi_1(X_i)$ such that $\tau = g\cdot \tau'$ and $r|_{\wt{X}^i_{\tau}} = g\cdot r' \cdot g^{-1}|_{\wt{X}^i_{\tau}}$, where $\wt{X}^i_{\tau}$ is a union of components of $\wt{X}\setminus \hat{X_i}$ with non-empty intersection with $\tau$ (see Figure \ref{figure_set_U}):
\[
\wt{X}^i_{\tau} := \bigcup_{h^i\in \pi_1(X_i)/\pi_1(Z):\: h^i\bar{X_i}\cap\tau\neq\emptyset} h^i\bar{X_i}.
\]
\begin{figure}[h]
\centering
\includegraphics[width=0.5\textwidth]{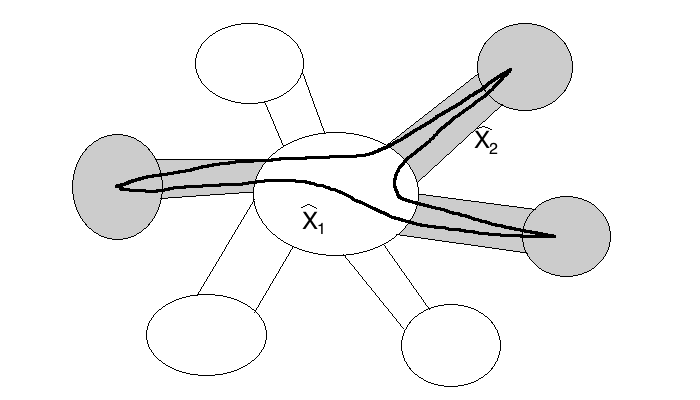}
\caption{The visible simplex is $\tau$, grey region corresponds to $\wt{X}^1_{\tau}$.}
\label{figure_set_U}
\end{figure}
Note that the sum on the right hand side of the above definition is finite. In particular, for every element $[\tau, r]\in \mathcal{U}^i(\sigma)$ the simplex $\pi\circ r(\tau)\in C(\Delta^k, X)$ is well defined. We have also an action of $\prod_{j\in\pi_1(X_i)/\pi_1(Z)}\pi_1(Z)$ on $\mathcal{U}^i(\sigma)$ defined as $A\cdot [\tau, r] = [\tau, A\cdot r]$. Now consider the sets
\[
\mathcal{S}:= \{\partial_l\sigma\::\: \sigma\in\supp(c),\, \partial_l\sigma\cap Z\neq\emptyset,\, l=0,...,k\}
\]
and
\[
\mathcal{U}:=\bigcup_{i=1}^2\bigcup_{\sigma\in\mathcal{S}}\mathcal{U}^i(\sigma).
\]
Consider also the following function $f:\mathcal{U}\rightarrow \mathbb{R}$. For $\sigma\in\supp(c)$, $l\in\{0,..,k\}$ such that $\partial_l\sigma\cap Z\neq\emptyset$ and $H^i_j$ such that $H^i_j\cdot\partial_l\theta(\sigma)\cap \hat{X}_i\neq\emptyset$ define
\[
f_{\sigma}([H^i_j\partial_l\theta(\sigma),r_i]) = (-1)^l c(\sigma)
\]
and $f_{\sigma}([\tau, r]) = 0$ for all other $[\tau, r]\in\mathcal{U}$. Finally, define
\[
f = \sum_{\sigma\in\supp(c)} f_{\sigma}.
\]
Note that $f$ is finitely supported. Indeed, the above sum is finite because $f_{\sigma}$ can be non-zero only if $\sigma\cap Z\neq \emptyset$ and every $f_{\sigma}$ is finitely supported because there are only finitely many $H^i_j$ such that $H^i_j\theta(\sigma)\cap \hat{X_i}\neq\emptyset$. The importance of $f$ is reflected by the fact that by the above construction we have
\[
(\partial D(c,r_1,r_2))|_{n(Y)}(\tau) = \sum_{\{[\tau',r]\in\mathcal{U} \::\: \pi(r(\tau')) = \tau  \}} f([\tau',r])
\]
for $\tau\in C(\Delta^{k-1}, X)$. Therefore
\[
|\partial (D(c, r_1, r_2)|_{n(Y)})|_1\leq \|f\|_1.
\]
Moreover, the above relations are invariant under the action of
\[
G:= \bigoplus^2_{i=1}\bigoplus_{j\in\pi_1(X_i)/\pi_1(Z)}\pi_1(Z),
\]
i.e. for $(g_1,g_2)\in G$ we have
\[
(\partial D(c,g_1\cdot r_1,g_2\cdot r_2))|_{n(Y)}(\tau) = \sum_{\{[\tau',r]\in\mathcal{U} \::\: \pi(r(\tau')) = \tau  \}} f((g_1,g_2)\cdot[\tau',r]).
\]
The final ingredient in the proof is the following claim.

\emph{Claim}: For every orbit $U$ of $G$-action on $\mathcal{U}$ we have $|\sum_{u\in U} f(u)|=0$.

Assuming that the above claim is true, let $\varepsilon>0$ and let $\mu\in\ell^1(G)$ be a probability measure given by Proposition \ref{prop_Gromov_diffusion} for the above function $f$ and the action of $G$. Consider the chain
\[
D'(c) := \mu*D(r_1,r_2, c) =  \sum_{(g_1, g_2)\in G} \mu((g_1,g_2))\cdot D(g_1\cdot r_1, g_2\cdot r_2, c).
\]
It is a finite convex linear combination of chains of the form $D(r'_1, r'_2, c)$ for $r'_i\in\mathcal{R}_i$ for $i=1,2$ hence it is obvious that it is a finite modification of $c$, $S^{(K)}D'(c) = S^{(K)}(c)$, $|D'(c)|^{ne(Z)}_1\leq |c|_1$ and $\supp((\partial D'(c))|_{n(Y)})$ is finite. Moreover, it follows that
\[
|(\partial D'(c))|_{n(Y)}|_1\leq \|\mu*f\|_1 \leq \varepsilon,
\]
hence the chain $D'(c)$ satisfies the required conditions.

\emph{Proof of the claim}. For a simplex $\sigma\in\mathcal{S}$ let $(\sigma_1,l_1) ,...,(\sigma_t, l_t)$ be the complete list of pairs of simplices in $\supp(c)$ and numbers in $\{0, ... ,k\}$ such that $\sigma = \partial_{l_s}\sigma_s$ for $s=1,...,t$. From the fact that $c$ is an image of a cycle in $C_*^{\lf}(X, Y\setminus Z)$ and $\sigma\in n(Y\setminus Z)$ it follows that
\[
\sum_{s=1}^t c(\sigma_s)(-1)^{l_s}=(\partial c)(\sigma) = 0.
\]
Let $H^i_{j_1},...,H^i_{j_t}\in \pi_1(X)$ be representatives of cosets $\pi_1(X_i)\setminus \pi_1(X)$ such that $H^i_{j_s}\partial_{l_s}\theta(\sigma_s)$ are in the same $\pi_1(X_i)$-orbit. In particular, because $\hat{X_i}$ and $\pi^{-1}(Z)$ are invariant under the action of $\pi_1(X_i)$, either all or none of the simplices $H^i_{j_s}\partial_{l_s}\theta(\sigma_s)$ for $s=1,...,t$ have non-empty intersection with both $\pi^{-1}(Z)$ and $\hat{X_i}$. It follows that
\[
\sum_{s=1}^t c(\sigma_s)(-1)^{l_s} f_{\sigma}([H^i_{j_s}\partial_{l_s}\theta(\sigma_s), r_i]) = 0.
\]
It suffices to show that $[\partial_{l_s}H^i_{j_s}\theta(\sigma_s), r_i]\in \mathcal{U}^i(\sigma)$ for $s=1,...,t$ are in the same $G$-orbit. We will prove that for any $r,r'\in\mathcal{R}_i$, $\tau\in C(\Delta^k, \wt{X})$ and $\gamma\in \pi_1(X_i)$, the elements $[\tau, r], [\gamma\tau, r']\in \mathcal{U}^i(\pi(\tau))$ are in the same $\bigoplus_{j\in\pi_1(X_i)/\pi_1(Z)}\pi_1(Z)$-orbit. Let $h^i_j\in \pi_1(X_i)$ for $j\in\pi_1(X_i)/\pi_1(Z)$ be some representatives of cosets in $\pi_1(X_i)/\pi_1(Z)$. Then by the definition of $\mathcal{R}_i$ for every $j\in\pi_1(X_i)/\pi_1(Z)$
\[
r|_{h^i_j\bar{X}_i} = (h_j a^i_j)\circ r_{\hat{Z}}\circ ((a_j^i)^{-1}(h^i_j)^{-1})
\]
for some $a^i_j\in\pi_1(Z)$. Similarly,
\[
(\gamma^{-1} r' \gamma)|_{h^i_j\bar{X}_i} = (h^i_j b_j^i)\circ r_{\hat{Z}}\circ ((b_j^i)^{-1}(h^i_j)^{-1})
\]
for some $b^i_j\in\pi_1(Z)$. Therefore for every finite subset $J\subset \pi_1(X_1)/\pi_1(Z)$ and
\[
A := \bigoplus_{j\in J} (a_j^i(b_j^i)^{-1}) \in \bigoplus_{j\in\pi_1(X_i)/\pi_1(Z)} \pi_1(Z)
\]
we have $A\cdot r = \gamma^{-1}r'\gamma$ on $\bigcup_{j\in J}h^i_j\bar{X}_i$, hence $[\tau, A\cdot r]\sim [\gamma\tau, r']$ for sufficiently large $J$. This finishes the proof of the claim.
\end{proof}

\begin{remark}\label{remark_weakening}
In the above considerations, we assumed that $Z$ is compact and aspherical. However, the proof works equally well with slightly weaker assumptions, namely:
\begin{itemize}
\item there are only finitely many simplices $\sigma\in\supp(c)$ which intersect $Z$;
\item there exists a cellular retraction $\wt{X}\rightarrow \hat{Z}$.
\end{itemize}
\end{remark}

\section{Superadditivity of the locally finite simplicial volume II}\label{section_higher_dim_trick}
In this section we modify slightly Proposition \ref{prop_chain_decomposition} and state the following.

\begin{prop}\label{prop_higher_dimensional}
Let $(X_1,Y_1)$ and $(X_2,Y_2)$ be two pairs of $CW$-complexes such that all connected components $Y'_i$ of $Y_i$ for $i=1,2$ are compact and $\im(\pi_1(Y'_i)\rightarrow \pi_1(X_i))$ are amenable. Let $Z_i\subset Y_i$ for $i=1,2$ be two connected components such that there is a homeomorphism $f:Z_1\rightarrow Z_2$. Assume moreover that $\pi_j(Z_1)=0$ for $j=2,...,k-2$ where $k\geq 2$ and $\pi_1(Z_1)$ injects both into $\pi_1(X_1)$ and $\pi_1 (X_2)$. Finally, let $c\in C_k^{\lf}(X_1\cup_f X_2, (Y_1\setminus Z_1)\cup (Y_2\setminus Z_2))$ be a cycle. Then for every $\varepsilon>0$ there exist cycles $c_1\in C^{\lf}_k(X_1, Y_1)$ and $c_2\in C^{\lf}_k(X_2, Y_2)$ which are finite approximations of restrictions of $c$ to $C_k^{\lf}(X_1\cup_f X_2, Y_1\cup X_2)$ and $C_k^{\lf}(X_1\cup_f X_2, X_1\cup Y_2)$ respectively, such that
\[
|c_1|^{ne(Z)}_1+|c_2|^{ne(Z)}_1 \leq |c|_1+\varepsilon.
\]
In particular, $\|[c_1]\|_1+\|[c_2]\|_1 \leq \|[c]\|_1$ by Proposition \ref{prop_amenable_subset}.
\end{prop}

It follows directly from the following lemma.

\begin{lemma}\label{lemma_higher_dim_trick}
Let $k\geq 3$ and let $X$ be a CW-complex with a subcomplex $Y\subset X$ such that all connected components $Y''$ of $Y$ are compact and $\im(\pi_1(Y'')\rightarrow \pi_1(X))$ are amenable. Let $n\geq k$, let $f:S^{n-1}\rightarrow Y$ be a cellular map and let $c\in C^{\lf}_k(X\cup_f D^n,Y\cup_f D^n)$. Then for every $\varepsilon>0$ there exists $Y'\subset Y\cup_f D^n$ which is a finite sum of connected components of $Y\cup_f D^n$ and a finite approximation $c'\in C^{\lf}_k(X,Y)$ of $c$ such that
\[
|c'|^{ne(Y')}\leq |c|^{ne(Y')}_1 + \varepsilon.
\]
In particular, by Proposition \ref{prop_amenable_subset}
\[
\|[c]\|_1^{(X,Y)} = \|[c]\|_1^{(X\cup_f D^n, Y\cup_f D^n)}.
\]
\end{lemma}

\begin{proof}[Proof of Proposition \ref{prop_higher_dimensional}]
Let $\bar{Z}$ be an aspherical space made from $Z_1$ by gluing some cells of dimension $\geq k$. Then by Proposition \ref{prop_chain_decomposition} and Remark \ref{remark_weakening} for every $\varepsilon'>0$ there exist cycles $c_1\in C_k^{\lf}(X_1\cup \bar{Z}, Y_1\cup\bar{Z})$ and  $c_2\in C_k^{\lf}(X_2\cup_f \bar{Z}, Y_2\cup_f\bar{Z})$ which are finite approximations of restrictions of $c$ to $C_k^{\lf}(X_1\cup \bar{Z}\cup_f X_2, Y_1\cup \bar{Z}\cup_f X_2)$ and  $C_k^{\lf}(X_1\cup \bar{Z}\cup_f X_2, X_1\cup \bar{Z}\cup_f Y_2)$ respectively such that
\[
|c_1|_1^{ne(\bar{Z})} + |c_2|_1^{ne(\bar{Z})} \leq |c|_1 +\varepsilon'.
\]
However, note that simplices in $\supp(c_1)$ and $\supp(c_2)$ can be easily modified such that they intersect only finitely many cells of $\bar{Z}$. Therefore we can assume that $\bar{Z}$ is made from $Z_1$ by gluing only finitely many cells. Now it suffices to prove the statement by induction on the number of added cells, which is obvious by Lemma \ref{lemma_higher_dim_trick}.
\end{proof}

The rest of this section is devoted to the proof of Lemma \ref{lemma_higher_dim_trick}. We start with the following two technical lemmas.

\begin{lemma}\label{lemma_amenable_inequality}
Let $G$ be a finitely generated amenable group with a set of generators $S$, let $\varepsilon>0$ and let $\mu\in \ell^1(G)$ be a finitely supported probability measure such that
\[
\|\mu-s\cdot\mu\|_1\leq \varepsilon
\]
for every $s\in S$. Let $\rho = \sum^n_{i=1} a_i g_i\in\mathbb{R}[G]$ be an element such that $\sum^n_{i=1} a_i = 0$. Then
\[
\|\rho\cdot \mu\|_1 \leq \varepsilon\cdot |\rho|_1\cdot \sup_i |g_i|,
\]
where $|g|$ is a word length of $g\in G$ with respect to $S$.
\end{lemma}

\begin{proof}
Let $g\in G$. Then $g=s_1\cdot ... \cdot s_{|g|}$ for some $s_j\in S$, $j=1,...,|g|$. We have
\[
\|\mu-g\cdot\mu\|_1\leq \sum^{|g|}_{j=1} \|(s_1\cdot ... \cdot s_{j-1})(\mu- s_j\cdot \mu)\|_1 \leq \varepsilon\cdot |g|.
\]
Using the above inequality, we compute
\[
\|\rho\cdot\mu\|_1 = \|\rho\cdot\mu - \sum^n_{i=1} a_i\mu\|_1 = \|\sum^n_{i=1} a_i(g_i\cdot\mu - \mu)\|_1 \leq \sum^n_{i=1} |a_i| \|g_i\cdot\mu - \mu\| \leq \varepsilon\cdot |\rho|_1\cdot \sup_i |g_i|.
\]
\end{proof}

\begin{lemma}\label{lemma_sphere_diffusion}
Let $n\geq 1$ and let  $c_0\in C_n(S^n)$ be a cycle. Then for every $\varepsilon>0$ there exists a chain $c\in C_{n+1}(S^n)$ such that
\begin{enumerate}
\item $\partial_0 c = c_0$;
\item $|c|_1\leq |c_0|_1$;
\item $|(\partial -\partial_0)c'|_1\leq \varepsilon$.
\end{enumerate}
\end{lemma}

\begin{proof}
Let $x\in S^n$ be any point. There exists an operator $c^x_k:C(\Delta^k,S^n)\rightarrow C(\Delta^{k+1},S^n)$ for $k=0,...,n-1$ such that
\begin{itemize}
\item $\partial_0c^x_k\sigma = \sigma$ and the $0$-vertex of $c^x_k$ is $x$;
\item $\partial_{i+1}c^x_k(\sigma) = c^x_k(\partial_i\sigma)$ for $k\geq 1$ and $i=0,...,k$.
\end{itemize}
One can easily construct this operator by induction on $k$. For $k=0$ simplices can be identified with points of $S^n$, so let $y\in S^n$. The second condition is empty, hence one can define $c_0^x(y)$ to be any interval joining $y$ and $x$. For $k>0$ and $\sigma\in C(\Delta^k, S^n)$, the faces $\partial_i c^x_k(\sigma)$ for $i=0,...,k+1$ are defined by induction. It follows from the second property that these faces define a map $\partial\Delta^{k+1}\rightarrow S^n$. Because $\pi_k(S^n)=0$ there is a simplex $\sigma'\in C(\Delta^{k+1},S^n)$ such that $\partial_i\sigma' = \partial_i c^x_k(\sigma)$ for $i=0,...,k+1$. Define $c_k^x(\sigma) = \sigma'$.

For $\sigma\in C(\Delta^{n-1},S^n)$ Let
\[
\Theta(\sigma) := \{\sigma'\in C(\Delta^n, S^n)\::\: \partial_i\sigma' = \partial_i c^x_{n-1}(\sigma) \text{ for }i=0,...,n  \}
\]
and let $\Xi(\sigma)\subset\Theta(\sigma)$ be some set of representatives of homotopy classes of simplices in $\Theta(\sigma)$ relative to their boundaries. Note that $\pi_n(S^n,x)$ acts freely and transitively on $\Xi(\sigma)$ in the similar way as $\pi_n(X)$ acts on $\pi_n(X,Y)$ for a pair of topological spaces $Y\subset X$. For an $n$-simplex $\sigma_0\in C(\Delta^n, S^n)$ let $\Omega(\sigma_0)$ be the set of sequences $(\sigma_0,...,\sigma_{n+1})$ of $n$-simplices such that $\sigma_i\in \Xi(\partial_{i-1}\sigma_0)$ for $i=1,...,n+1$ and define
\[
\omega(\sigma_0) = \{(\sigma_0,...,\sigma_{n+1})\in \Omega(\sigma_0)\::\: \exists_{\sigma\in C(\Delta^{n+1},S^n)} \forall_{i=0,...,n+1}\, \partial_i\sigma = \sigma_i \}.
\]
In other words, $\omega(\sigma_0)$ is a set of sequences of $n+1$ simplices (from which the first one is $\sigma_0$ and the rest is 'standarized') such that they are consequent faces of some $n+1$-simplex (see Figure \ref{figure_sphere_lemma}).  There is an equivalent description of $\omega(\sigma_0)$. Namely, there is a map
\begin{eqnarray*}
h_{\sigma_0}: \Omega(\sigma_0) & \rightarrow & \pi_n(S^n);  \\
h((\sigma_0,....,\sigma_{n+1}))|_{\partial_i\Delta^n} & = & \sigma_i
\end{eqnarray*}
and $\omega(\sigma_0) = (h_{\sigma_0})^{-1}(0)$. Note that for $i=1,...,n+1$ there is an action $\rho_i$ of $\pi_n(S^n)$ on $\Omega(\sigma_0)$ defined as
\[
\rho_i(\xi)\cdot (\sigma_0,....,\sigma_{n+1}) = (\sigma_0,...,\xi\cdot \sigma_i,...,\sigma_{n+1})
\]
and $h_{\sigma_0}$ is equivariant with respect to this action. It follows that for the action $\rho = \bigoplus_{i=1}^{n+1}\rho_i$ of $\bigoplus^{n+1}_{i=1}\pi_n(S^n)$ on $\Omega(\sigma_0)$ we have
\[
h_{\sigma_0}((\xi_1,...,\xi_{n+1})\cdot (\sigma_0,....,\sigma_{n+1})) = h_{\sigma_0}((\sigma_0,\xi_1\cdot\sigma_1....,\xi_{n+1}\cdot\sigma_{n+1})) = (\prod_{i=1}^{n+1}\xi_i)\cdot h_{\sigma_0}(\sigma_0,...,\sigma_{n+1})
\]
hence the group
\[
G = \{(\xi_i)^{n+1}_{i=1}\in \prod^{n+1}_{i=1}\pi_n(S^n,x)\::\: \sum^{n+1}_{i=1} \xi_i = 0 \}
\]
acts on $\omega(\sigma_0)$ and the action is transitive when composed with a projection onto $i$-th factor of $s\in\omega(\sigma_0)$ for $i=1,...,n+1$.

\begin{figure}[h]
\centering
\includegraphics[width=0.5\textwidth]{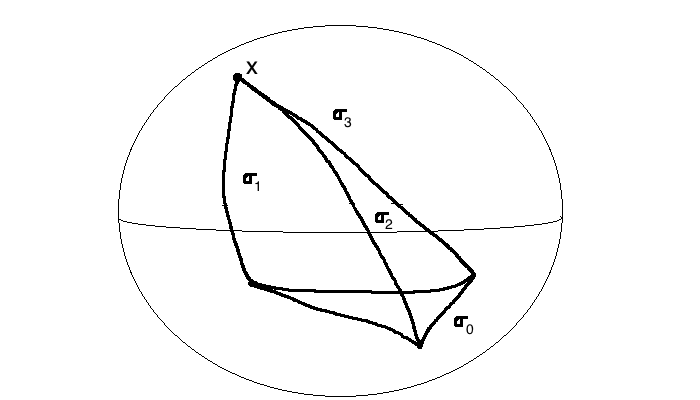}
\caption{Example of an element $(\sigma_0,...,\sigma_3)$ of $\omega(\sigma_0)$. $F((\sigma_0,...,\sigma_3))$ can be defined as any $3$-simplex that 'fills' $(\sigma_0,...,\sigma_3)$.}
\label{figure_sphere_lemma}
\end{figure}

Let $c_0 = \sum^m_{j=1} a_j\sigma_0^j$. For $j=1,...,m$ choose a sequence $s^j = (\sigma_0^j,...,\sigma^j_{n+1})\in\omega(\sigma^j_0)$ and consider the chain
\[
c' = \sum_{j=1}^m a_j F(s^j) \in C_{n+1}(S^n),
\]
where $F:\omega(\sigma_0)\rightarrow C(\Delta^{n+1},S^n)$ is any map assigning to a sequence $s\in\omega(\sigma_0)$ an $n+1$ simplex $\sigma(s)$ such that $\partial_i\sigma(s) = \sigma_i$ for $i=0,...,n+1$. It is clear that $\partial_0 c' = c_0$ and $|c'|_1\leq |c_0|_1$, but the norm of $(\partial-\partial_0)c'$ may be large. However, we can use the amenability of $G$ to modify $c'$ in order to decrease this norm. Let $\mathcal{C}\subset C_{n+1}(S^n)$ be the set of singular chains generated by the simplices of the form $F(s)$ for $s\in\omega(\sigma)$ and $\sigma\in\supp(c_0)$. Note that $\bigoplus^m_{j=1} G$ acts on $\mathcal{C}$ simplex-wise by
\[
(g_1,....,g_m)\cdot F(s) = F(g_j\cdot s),
\]
where $s\in\omega(\sigma_0^j)$. Let $S\subset G$ be a finite symmetric set of generators which is invariant under the action of $S_{n+1}$ on $\prod^{n+1}_{i=1}\pi_n(S^n)$ permuting the indices. Because $G$ is amenable, there exists a finitely supported probability measure $\mu\in \ell^1(G)$ such that for every $g\in S$
\[
\|g\cdot\mu - \mu\|_1\leq \varepsilon',
\]
where $\varepsilon'>0$ is fixed, but we will choose its value later. By averaging along every orbit, we can also assume that $\mu$ is invariant under the action of $\Sigma_{n+1}$ on $G\subset \prod^{n+1}_{i=1}\pi_n(S^n,x)$. We will denote by $\mu'\in\ell^1(\mathbb{Z})$ the push-forward of $\mu$ to any of groups $\pi_n(S^n,x)$ present in the above product. We claim that the chain
\[
c = \sum_{j=1}^m a_j (\mu* F(s^j))
\]
satisfies the desired conditions.

Because taking $\partial_0$ for simplices in $\mathcal{C}$ is invariant under the action of $\bigoplus^m_{j=1} G$ we have $\partial_0c = \partial_0 c' = c_0$. Similarly, $|c|_1\leq |c'|_1\leq |c_0|_1$. Let us estimate the norm of $c'' = (\partial-\partial_0)c$. By construction, $\supp(c'')$ is contained in $\cup_{\sigma\in c_0^{(n-1)}}\Xi(\sigma)$, hence
\[
|c''|_1\leq \sum_{\sigma\in c_0^{(n-1)}}|c''|^{\Xi(\sigma)}_1.
\]
Let $\sigma\in c_0^{(n-1)}$. Then $(\partial-\partial_0)c'|_{\Xi(\sigma)}$ is of the form
\[
\sum_{l=1}^{m(\sigma)} f_l\cdot\tau_l,
\]
where $m(\sigma)$ is a number of simplices in $\supp(c_0)$ with some face in $\Xi(\sigma)$ (with repetitions if any simplex in $\supp(c_0)$ has non-distinct faces), $\tau_l\in\Xi(\sigma)$ for $l=1,...,m(\sigma)$ and $f_l$ are coefficients such that $\sum_{l=1}^{m(\sigma)} f_l = 0$ because $c_0$ is a cycle. Therefore $c''|_{\Xi(\sigma)}$ is of the form
\[
\sum_{l=1}^{m(\sigma)} f_l(\mu*\tau_l),
\]
Let $\xi^{\sigma}_2,...,\xi^{\sigma}_{m(\sigma)}\in\pi_n(S^n,x)\cong \mathbb{Z}$ be elements such that $\xi^{\sigma}_l\cdot \tau_l = \tau_1$ for $l=2,...,m(\sigma)$. By Lemma \ref{lemma_amenable_inequality} we have
\[
|c''|^{\Xi(\sigma)}_1 \leq \varepsilon'\cdot |c_0|_1\sup_{l}\| \xi^{\sigma}_l\|.
\]

Finally, we set
\[
\varepsilon' = \frac{\varepsilon}{|c_0^{(n-1)}| \cdot |c_0|_1\sup_{\sigma,l}\|\xi^{\sigma}_l\|}
\]
and compute
\[
|c''|_1\leq \sum_{\sigma\in c_0^{(n-1)}}|c''|^{\Xi(\sigma)}_1 \leq \sum_{\sigma\in\ c_0^{(n-1)}} \varepsilon'\cdot |c_0|_1\sup_{\sigma,l}\|\xi^{\sigma}_l\| \leq \varepsilon.
\]
\end{proof}
Note that in the above Lemma we will have in fact $|c|_1=|c_0|_1$, because $\partial_0$ is a bounded operator of norm $1$. It follows that every simplex in $\supp(c)$ corresponds to $\partial_0\sigma\in\supp(c_0)$.

Now, we move more directly toward the proof of Lemma \ref{lemma_higher_dim_trick}. For two topological spaces $U, V$ we denote by $U\star V$ their join, i.e the space $U\times V\times I/ \sim$, where $\sim$ is generated by the relations $(u,v,0)\sim (u,v',0)$ and $(u,v,1)\sim (u',v,1)$ for $u,u'\in U$ and $v,v'\in V$. In particular, we have
\[
\Delta^k\star \Delta^l \cong \Delta^{k+l+1}
\]
for every $k,l\in\mathbb{N}$. For two maps $f:U\rightarrow W$ and $g:V\rightarrow W$, where $W$ is a geodesic metric space with unique geodesics, we define a \emph{geodesic join} of $f$ and $g$ to be a map $f\star g:U\star V\rightarrow V$ defined by
\[
f\star g(u,v,t) = [f(u), g(v)](t),
\]
where $[p,q]$ for $p,q\in W$ is a unique geodesic joining $p$ and $q$. If $U = \Delta^k$ and $V=\Delta^l$ are simplices then $f\star g$ can be considered canonically as a $k+l+1$-simplex with the ordered set of vertices being respectively the ordered list of vertices of $f$ followed by the ordered list of vertices of $g$.

Consider a sphere $S^{n-1}$ viewed as an equator of a sphere $S^n$, which is a boundary of a ball $D^{n+1}$. We have then an injective map
\begin{eqnarray*}
\beta: C(\Delta^k,S^{n-1}) & \rightarrow & C(\Delta^{k+2},D^{n+1}) \\
\sigma \mapsto E\star \sigma.
\end{eqnarray*}
Here, $E:I\rightarrow D^{n+1}$ is a $1$-simplex defined to be the interval joining the poles of $D^{n+1}$ and perpendicular to the hyperplane spanned by $S^{n-1}$ (see Figure \ref{figure_beta_operator}). We will denote also by $\beta$ its extension to a map $C_k(S^{n-1})\rightarrow C_{k+2}(D^{n+1})$. Note that we have
\[
\partial\beta = \partial_0\beta - \partial_1\beta + \beta\partial.
\]

\begin{figure}[h]
\centering
\includegraphics[width=0.5\textwidth]{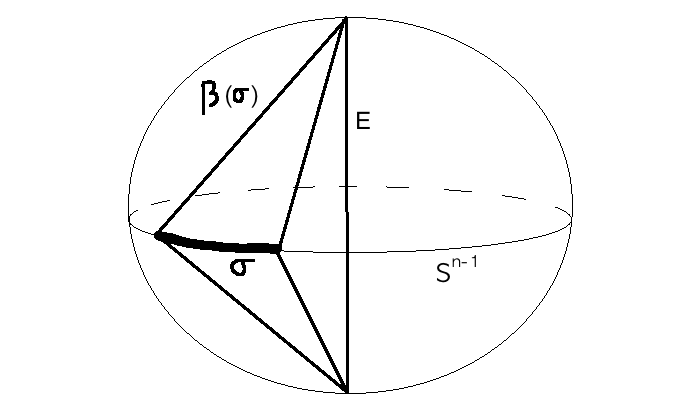}
\caption{A simplex $\sigma\in C(\Delta^k, S^{n-1})$ and its image under $\beta$.}
\label{figure_beta_operator}
\end{figure}

\begin{defi}
A singular chain $c\in C_*(D^{n+1})$ is called \emph{biconical} if it is in the image of $\beta$ up to permuting the vertices of simplices in $\supp(c)$.
\end{defi}

\begin{lemma}\label{lemma_point_on_edge}
Let $c\in C_n^{\lf}(X\cup_f D^n, Y\cup_f D^n)$ be a cycle. Then there exists a point $p\in int(D^n)$ such that $c$ is a finite approximation of a chain $c'$ such that $|c'|^{ne(Y')}_1 \leq |c|^{ne(Y')}_1$ for some set $Y'\subset Y\cup_f D^n$ with finitely many path-connected components, and for every $\sigma\in\supp(c')$ containing $p$, $\sigma^{-1}(p)$ is a discrete set contained in the edges of $\sigma$ which are not contained in $Y\cup_f D^n$. Moreover, we can assume that there exists $N\in\mathbb{N}$, a neighbourhood $U\subset D^n$ of $p$ and a diffeomorphism $\rho: U\rightarrow D^n$ such that for every $\sigma\in\supp(S^{(N)}(c'))$ containing $p$ the simplex $\rho\circ\sigma$ is biconical.
\end{lemma}

\begin{proof}
By smoothing $c$ locally (using Whitney approximation theorem \cite[Theorem 6.19]{LeeS} and Lemma \ref{lemma_local_chain_homotopy}) we can assume it is smooth when restricted to some open subset $V\subset D^n$. Using Sard's theorem \cite[Theorem 6.8]{LeeS} choose $p\in V$ to be any regular value of all $\sigma\in\supp(c)$ which is not in $c^{(n-1)}$. Then for every $\sigma\in\supp(c)$ the set $\sigma^{-1}(p)$ is a finite set of points in the interior of $\Delta^n$ and $\sigma$ is a local diffeomorphism in some neighbourhood of $\sigma^{-1}(p)$. Let $Y'\subset Y\cup_f D^n$ be the (finite) sum of connected components of $Y\cup_f D^n$ with non-empty intersection with any $\sigma\in\supp(c)$ containing $p$. We will construct the chain $c'$ in three steps. The first one is to modify $c$ to a chain $c''$ such that every simplex $\sigma\in\supp(c'')$ containing $p$ has some edge not contained in $Y\cup_f D^n$. The second step is to modify $c''$ to $c'''$ such that $p$ lies only on the edges of $c'''$ not contained in $Y'$, and in a 'regular' way. Finally, the third one is to disturb $c'''$ to $c'$ in order to satisfy 'local biconical' condition.

We start with the construction of a chain $c''$. We apply Lemma \ref{lemma_local_chain_homotopy} to modify $c$ such that for every $\sigma\in\supp(c)$
\begin{itemize}
\item if $\sigma\cap Y'\neq \emptyset$ and $\sigma$ has some edge not contained in $Y'$ then $\sigma$ is $Y'$-barycentrically non-degenerated;
\item if $\sigma$ contains $p$ then every $\sigma'\in\supp(S_{Y'}(\sigma))$ containing $p$ has some edge not contained in $Y'$.
\end{itemize}
For simplices $\sigma\in\supp(c)$ with all edges in $Y'$ the second condition can be realised simply by modifying $\sigma$ such that its barycenter does not lie in $Y'$. Note also that for these simplices $S_{Y'}(\sigma)\subset e(Y')$. We set $c'' = S_{Y'}(c)$. Then $c''$ is a finite approximation of $c$, $|c''|^{ne(Y')}_1\leq |c|^{ne(Y')}_1$ by Proposition \ref{prop_local_barycentric_subdivision} and every $\sigma\in\supp(c'')$ containing $p$ has some edge not contained in $Y\cup_f D^n$ by the definition of $c''$.

To construct the chain $c'''$ we use again Lemma \ref{lemma_local_chain_homotopy}, but the description is more complicated. To be more clear, we define a system of homotopies $H^{\sigma}_k:\Delta^k\times I$ for $\sigma\in C(\Delta^k, X\cup_f D^n)$ and $k=0,...,n$. Choose some simplex $\tau\in\supp(c'')$ containing $p$, $y\in\tau^{-1}(p)$ and an edge $e$ of $\tau$ not contained in $Y\cup_f D^n$. Moreover, let $\gamma:[-1,1]\rightarrow D^n$ be some smooth path such that $\gamma(0)=p$ and $\frac{d\gamma}{dt}(0) = 1$. Then there is a homotopy $H^{\tau}_n:\Delta^n\times I\rightarrow X\cup_f D^n$ such that
\begin{itemize}
\item $H^{\tau}_n(\cdot, 0)=\tau$;
\item $H^{\tau}_n$ is constant on the vertices, some neighbourhood of $\tau^{-1}(p)\setminus\{y\}$ and on the faces of $\tau$ not containing $e$;
\item $H^{\tau}_n(\cdot, 1)|_{e}$ is a smooth path such that for every $t\in[0,1]$ such that $e(t)=p$ we have
\[
e|_{(t-\varepsilon, t+\varepsilon)} = \gamma|_{(-\varepsilon, \varepsilon)}
\]
for some $\varepsilon>0$;
\item $|H^{\tau}_n(\cdot, t)^{-1}(p)| = |\tau^{-1}(p)|$ for $t\in [0, 1]$;
\item $|int(\Delta^n)\cap H^{\tau}_n(\cdot, 1)^{-1}(p)| = |\tau^{-1}(p)|$ for $t\in [0,1)$ and $|int(\Delta^n)\cap H^{\tau}_n(\cdot, 1)^{-1}(p)| = |\tau^{-1}(p)|-1$.
\end{itemize}

\begin{figure}[h]
\centering
\includegraphics[width=0.5\textwidth]{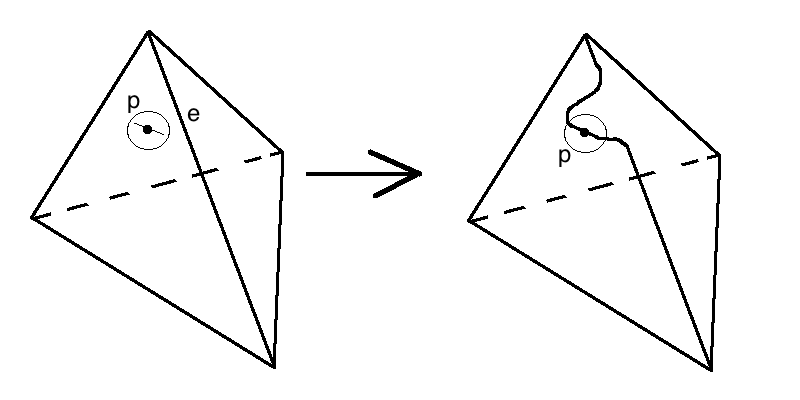}
\caption{An example of a homotopy $H^{\tau}_n$ for a simplex containing $p$ (in the simplest case when no edges of $\tau$ are identified).}
\label{figure_homotopy_example}
\end{figure}

In other words, $H^{\tau}_n$ is a homotopy modifying $\tau$ in order to 'move' one of the points in $\tau^{-1}(p)$ to some edge of $\tau$ not contained in $Y\cup_f D^n$, but fixing the others (see Figure \ref{figure_homotopy_example}). Having defined $H^{\tau}_n$ (and, implicitly, $H^{\sigma}_k$ for every $k$-face $\sigma\subset\tau$) we define $H^{\sigma}_k$ for $\sigma\in C(\Delta^k, X\cup_f D^n)$ and $k=0,...,n$ by induction on $k$. For $k=0$ we set $H^{\sigma}_k$ to be constant homotopies, as well as for $k=1$ except the edge $e$. Having $H^{\sigma}_l$ defined for $l < k$, we define $H^{\sigma}_k$ simply as homotopy extensions of corresponding homotopies of the boundaries, choosing however homotopies with the following properties.
\begin{itemize}
\item $H_k^{\sigma}$ is constant if $H_{k-1}^{\partial_i\sigma}$ is constant for $i=0,...,k$;
\item $|int(\Delta^k)\cap H_k^{\sigma}(\cdot,t)^{-1}(p)|$ is constant with respect to $t$ except for $t=1$, when it may decrease by $1$ if $\sigma$ has the edge $e$.
\item if $\sigma\neq\tau$ then $H^k_{\sigma}$ is constant in some neighbourhood of $\sigma^{-1}(p)$.
\end{itemize}
We apply Lemma \ref{lemma_local_chain_homotopy} using the above system of homotopies and repeat this process for all points $y\in \tau^{-1}(p)$ and all $\tau\in\supp(c'')$ containing $p$. Note that there are only finitely many such simplices and points, and each chain homotopy modify only finitely many simplices, hence the resulting chain $c'''$ is a finite approximation of $c$. It has also the property that for each $\sigma\in\supp(c''')$, $\sigma^{-1}(p)$ is contained in the 1-skeleton of $\sigma$, and each edge of $\sigma$ passing through $p$ is smooth and passes through $p$ with velocity 1.

The last step is to 'regularize' the behaviour of $c'''$ around $p$. Let $t\in (0,1)$ and let $e$ be an edge in $(c''')^{(1)}$ such that $e(t)=p$. Choose sufficiently small $\varepsilon$ and a diffeomorphism $\rho_{e,t}: U\rightarrow D^n$ defined on some neighbourhood $U$ of $p$ such that $\rho(e|_{[t-\varepsilon, t+\varepsilon]})$ is the interval joining two poles of $D^n$. Let also $N\in\mathbb{N}$ be a number such that the simplices in $S^{(N)}(c''')$ containing $e(t)$ are contained in $U$. Note that because there are only finitely many such pairs $(e,t)$, we can choose $N$ uniformly for all of them. We can also assume (by reparametrising $e$ slightly if necessary) that $e(t)$ is not a vertex of $S^{(N)}(c''')$.

Let $c'''_{e,t}$ be a subchain of $c'''$ consisting of simplices in $\supp(S^{(N)}(c'''))$ containing $e(t)$. Without loss of generality we can assume that every $\sigma\in\supp(c'''_{e,t})$ has distinct vertices. Consider $D^n$ with a CW-structure consisting of one $n$-cell glued to $S^{n-1}$, two $n-1$-cells glued to the equator $S^{n-2}\subset S^{n-1}$ and lower dimensional cells which correspond to $S^{n-2}$. Using cellular approximation we can homotopy $c'''_{e,t}$, keeping edges containing $e(t)$ fixed, such that for every $\sigma\in\supp(c'''_{e,t})$ every codimension-2 face $\sigma'\subset \sigma$ spanned by the vertices not contained in the edge containing $e(t)$, the simplex $(\rho_{e,t})_*(\sigma')$ is contained in the equator $S^{n-2}$ of the sphere $S^{n-1}\subset D^n$. Moreover, because $D^n$ is convex, we can correct this homotopy (reparametrising $e$ slightly again if necessary) such that $\rho(c'''_{e,t})$ is biconical simply by joining the corresponding points by shortest geodesics. We repeat this procedure to all pairs $(e,t)$ such that $e(t)=p$ and obtain a cycle $c'$, satisfying the required conditions.
\end{proof}

To finish the proof of Lemma \ref{lemma_higher_dim_trick} we need only the following, easy lemma.

\begin{lemma}\label{lemma_barycentric_edge}
Let $N\in\mathbb{N}$, let $e$ be some edge of $\Delta^n$ and let $e'\in (S^{(N)}\Delta^n)^{(1)}$ be some subedge of $e$. Then the simplicial subcomplex $S^{(N)}_{e'}$ of $S^{(N)}\Delta^n$ consisting of the simplices containing $e'$ is of the form $e'\star K$, where $K$ is isomorphic (as a simplicial complex) to $S^{(N)}\Delta^{n-2}$.
\end{lemma}

\begin{proof}
It suffices to prove the result for $N=1$, because for larger $N$ it follows from this case by simple induction. Note that the simplices in $S\Delta^n$ are indexed by the ascending sequences of faces of $\Delta^n$ of length $n+1$, i.e. sequences $(\Delta_0,...,\Delta_n)$, where each $\Delta_k$ is some $k$-dimensional face of $\Delta^n$ for $k=0,...,n$ such that $\Delta_{k-1}\subset\Delta_k$ for $k=1,...,n$. The vertices of $(\Delta_0,...,\Delta_n)$ are the consequent barycenters of $\Delta_0,...,\Delta_n$. It follows that simplices in $S\Delta^n$ containing the edge $e'$ which is a subedge of some original edge of $\Delta^n$, are of the form $e'\star \Delta'$, where $\Delta'$ are indexed by ascending sequences $(\Delta_2,...,\Delta_n)$, where $\Delta_k$ is a $k$-dimensional face containing $e$ for $k=2,...,n$. In other words, they are indexed by ascending sequences $(e'\star \Delta'_0,...,e'\star \Delta'_{n-2})$, where $\Delta'_k$ for $k=0,...,n-2$ is a $k$-face of an $n-2$-dimensional face of $\Delta^n$ disjoint from $e$. A combinatorial isomorphism between $S_{e'}$ and $e'\star S\Delta^{n-2}$ is given then by
\[
(e'\star\Delta'_0,...,e'\star\Delta'_{n-2}) \mapsto e'\star(\Delta'_0,...,\Delta'_n),
\]
ant the inverse by 
\[
e'\star(\Delta'_0,...,\Delta'_n) \mapsto (e'\star\Delta'_0,...,e'\star\Delta'_{n-2}).
\]
\end{proof}

\begin{proof}[Proof of Lemma \ref{lemma_higher_dim_trick}]
Let $c\in C_k^{\lf}(X\cup_f D^n , Y\cup_f D^n)$. If $k<n$ then $c$ can be finitely approximated by a chain in $C^{\lf}_k(X,Y)$ by approximating it (locally) by a cellular chain, so the result is obvious. From now on we assume that $k=n$.

Let $p\in D^n$, $U\subset D^n$, $N\in\mathbb{N}$, $\rho:U\rightarrow D^n$, $Y'\subset Y$ and $c'\in C_n^{\lf}(X\cup_f D^n,Y\cup_f D^n)$ be given by Lemma \ref{lemma_point_on_edge} applied to $c$ and let $e$ be an edge of $S^{(N)}(c')$ containing $p$. Note that in fact there is at most one such edge, and for every $\sigma\in\supp(c')$ we have $\sigma^{-1}(e)\subset(\Delta^n)^{(1)}$. For a simplex $\sigma\in\supp(c')$ let 
\[
\mathcal{S}^{\sigma}_e := \{\Delta'\in\supp(S^{(N)}\Delta^n)\::\: \sigma_*(\Delta') \text{ has edge }e \}.
\]
Note that by Lemma \ref{lemma_barycentric_edge}, $\mathcal{S}^{\sigma}_e$ is a disjoint sum of subcomplexes $e'_1\star S^{(N)}\Delta''_1,..., e'_{l(\sigma)}\star S^{(N)}\Delta''_{l(\sigma)}$, where $\sigma_*(e'_i) = e$ for $i=1,...,l(\sigma)$. Let
\[
\mathcal{T}_e^{\sigma} := \sum_{i=1}^{l(\sigma)} \sigma_*(e'_i\star\Delta''_i).
\] 
Note that by the definition of $\mathcal{T}_e^{\sigma}$ and Lemma \ref{lemma_point_on_edge}, the chain $\rho_*(\mathcal{T}^{\sigma}_e)$ is biconical, i.e.
\[
\rho_*(\mathcal{T}^{\sigma}_e) = \beta(c^{\sigma}_e)
\]
for some $c^{\sigma}_e\in C_*(S^{n-2})$ (see Figure \ref{figure_operator_beta_definitions}). Consider the chain
\[
c_e := \sum_{\sigma\in\supp(c')} c'(\sigma)\cdot c^{\sigma}_e.
\]
It is a cycle. Indeed, if $\partial c_e\neq 0$, then $\partial c'$ would contain terms corresponding to $\rho^{-1}_*\beta(\partial c_e)$ (which are not contained in $C^{\lf}_*(Y)$, because no edge in $c'^{(1)}$ containing $e$ as some of its subdivided pieces is contained in $Y$), therefore would be also non-zero.

\begin{figure}[h]
\centering
\includegraphics[width=0.5\textwidth]{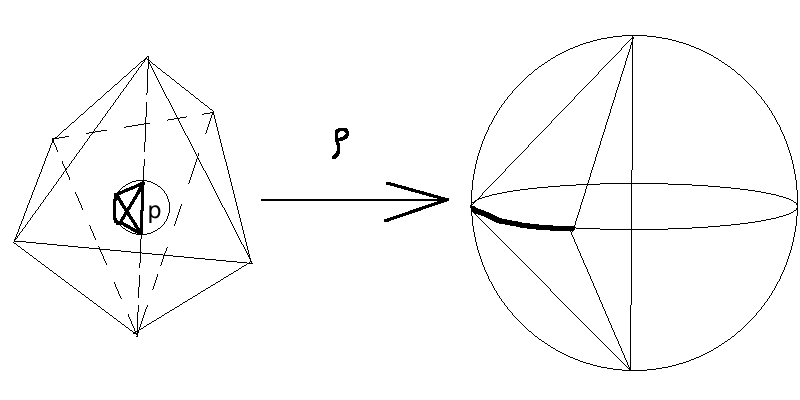}
\caption{On the left, examples of simplices in $\supp(c')$; an example of a simplex in $\supp(\mathcal{T}^{\sigma}_e)$ is made bold. On the right, example of a simplex in $\supp(\rho_*\mathcal{T}^{\sigma}_e)$; a simplex in $\supp(c^{\sigma}_e)$ is made bold.}
\label{figure_operator_beta_definitions}
\end{figure}

We view $D^{n-1}\subset D^n$ as a subset consisting of a disc bounded by the equator $S^{n-1}\subset D^n$. Consider an operator $\eta: C_*(D^{n-1})\rightarrow C_{*+1}(D^n)$ defined as
\[
\eta(\sigma)(t_0,..., t_{k+1}) = \begin{cases}
((0,...,-1)\star\sigma)(t_0 -t_1, 2t_1,t_2,..., t_{k+1}) \text{ for }t_0\geq t_1;\\
((0,...,1)\star\sigma)(t_1-t_0, 2t_0, t_2,...,t_{k+1}) \text{ for } t_0\leq t_1;
\end{cases}
\]
where $(0,...,-1)$ and $(0,...,1)$ are the south and the north pole of $D^n$ respectively. In other words, $\eta(\sigma)$ is formed from two cones over $\sigma$ with vertices $(0,...,-1)$ and $(0,...,1)$ glued along their common bases (see Figure \ref{figure_eta_operator}). Moreover, $\eta(\sigma)$ is arranged such that the south pole is $0$-th vertex and the north pole is $1$-st vertex. Note that if $\sigma$ does not contain the center of $D^n$, neither does $\eta(\sigma)$. Note also that
\[
\partial_0\eta(\sigma) = (0,....,1)\star\partial_0\sigma = \partial_0\beta(\partial_0\sigma),
\]
\[
\partial_1\eta(\sigma) = (0,....,-1)\star\partial_0\sigma = \partial_1\beta(\partial_0\sigma)
\]
and $\partial_i\eta(\sigma) = \eta(\partial_{i-1}\sigma)$ for $i=2,...,k+1$. In particular, we have
\[
\partial\eta = \partial_0\beta\partial_0 - \partial_1\beta\partial_0 -\eta(\partial-\partial_0) = (\partial\beta -\beta\partial)\partial_0 - \eta(\partial-\partial_0).
\]
\begin{figure}[h]
\centering
\includegraphics[width=0.5\textwidth]{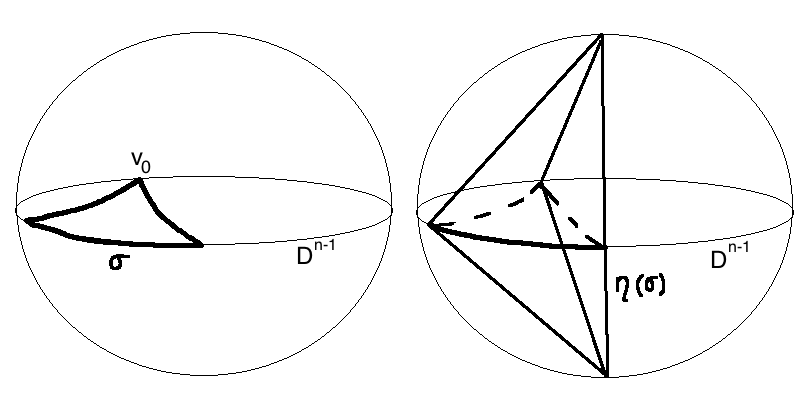}
\caption{Example of an application of $\eta$ to a simplex $\sigma$.}
\label{figure_eta_operator}
\end{figure}

Let $c'_e$ be a chain constructed by Lemma \ref{lemma_sphere_diffusion} applied to $c_e$ and some $\varepsilon'>0$, which we will choose later. For $\sigma\in\supp(c_e)$ we also define
\[
c'^{\sigma}_e := \sum_{\sigma'\in\supp(c'_e)\::\: \partial_0\sigma' = \sigma}c'_e(\sigma')\cdot\sigma',
\]
i.e. $c'^{\sigma}_e$ is a part of $c'_e$ that corresponds to a given $\sigma\in\supp(c_e)$. Finally, consider the chain
\[
c'' = \sum_{\sigma\in\supp(c')}c'(\sigma)\cdot \zeta_\sigma,
\]
where $\zeta_\sigma\in C_*(X)$ is constructed by cutting out $\mathcal{T}^{\sigma}_e$ from $\sigma$ and for every $\sigma'\in\supp(\mathcal{T}_e^{\sigma})$ gluing back in its place a chain
\[
\frac{\mathcal{T}_e^{\sigma}(\sigma')}{\sum_{\sigma''\in\supp(c'^{\rho_*\sigma'}_e)}c'_e(\sigma'')}\rho^{-1}_*\eta(c'^{\rho_*\sigma'}_e),
\]
i.e. the chain $\rho^{-1}_*\eta(c'^{\rho_*\sigma'}_e)$ normalized such that the sum of its coefficients is $\mathcal{T}_e^{\sigma}(\sigma')$ (see Figure \ref{figure_zeta_operator}).
\begin{figure}[h]
\centering
\includegraphics[width=1\textwidth]{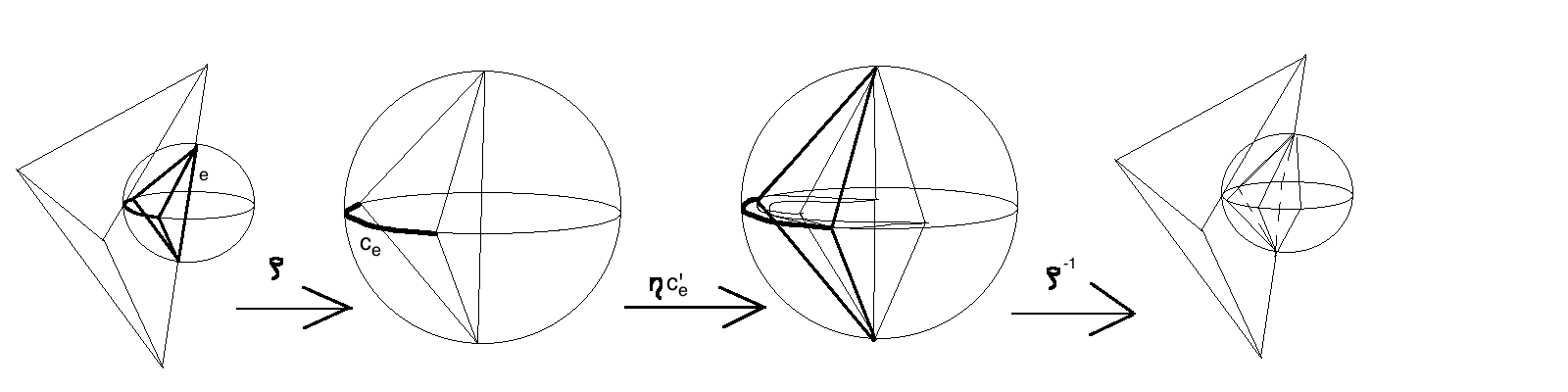}
\caption{Example of an application of $\zeta$ to a simplex $\sigma$.}
\label{figure_zeta_operator}
\end{figure}
We check that this 'definition' is correct. Note that a simplex $\sigma'\in \supp(\mathcal{T}^{\sigma}_e)$ glues to the rest of $\sigma$ by its $0$-th and $1$-st boundary components and these are the same as for any simplex $\sigma''\in\supp(\rho^{-1}_*\eta(c'^{\rho_*\sigma'}_e))$. To see this, let $\rho_*\sigma' = \beta(\tau')$ and let $\rho_*\sigma'' = \eta(\tau'')$. Then it follows that $\partial_0\tau'' = \tau'$. We compute
\[
\partial_0\sigma'' = \rho_*\partial_0(\eta\tau'') = \rho_*\partial_0\beta\partial_0\tau'' = \rho_*\partial_0\beta\tau' = \partial_0\sigma'.
\]
The same applies for $\partial_1$. Therefore every simplex $\sigma''\in\supp(\rho^{-1}_*\eta(c'^{\rho_*\sigma'}_e))$ can be glued in the place of $\sigma'\in \supp(\mathcal{T}^{\sigma}_e)$. Moreover, because the coefficients are normalized, for every $\sigma'\in \supp(\mathcal{T}^{\sigma}_e)$ we glue in its place a chain of the same weight. It follows that $c''\in C^{\lf}_n(X\cup_f(D^n\setminus \{p\}), Y\cup_f(D^n\setminus \{p\}))$ is a proper chain and is a finite modification of $c'$ (because there are finitely many simplices which contain $p$ and only these are modified). It has the following properties.
\begin{itemize}
\item $|c''|^{ne(Y')}_1\leq |c'|^{ne(Y')}_1$;
\item $c'''$ does not contain $p$;
\item $S^{(N)}(c'') = S^{(N)}(c')$ in  $C^{\lf}_*(X\cup_f D^n,Y\cup_f D^n)$;
\item $\partial S^{(N)}(c'')\in C^{\lf}_*(Y)$;
\item $|(\partial c'')|_{n(Y)}|_1\leq |\eta((\partial-\partial_0)c'_e)|_1\leq \varepsilon'$.
\end{itemize}
The first and second properties are obvious from the construction, similarly as the third and fourth, because all modified parts of simplices $\sigma\in\supp(c')$ are sums of simplices of $S^{(N)}(c')$ that are contained in $C_*(Y\cup_f D^n)$. The last one follows from the fact that if we look globally at the construction of $c''$, we cut out $\sum_{\sigma\in \supp(c')}c'(\sigma)\cdot\mathcal{T}^{\sigma}_e$ from $c'$ and glue back $\rho_*^{-1}\eta(c'_e)$. Moreover, we have
\begin{eqnarray*}
\partial\eta(c'_e) & = & (\partial\beta -\beta\partial)(\partial_0c'_e) - \eta((\partial-\partial_0)c'_e) \\
& = & (\partial\beta -\beta\partial)(c_e) - \eta((\partial-\partial_0)c'_e) \\
& = & \partial\beta(c_e) - \eta((\partial-\partial_0)c'_e) \\
& = & \partial\rho_*(\sum_{\sigma\in\supp(c')}c'(\sigma)\cdot\mathcal{T}^{\sigma}_e) - \eta((\partial-\partial_0)c'_e).
\end{eqnarray*}
Because $c'$ is a cycle in $C_*^{\lf}(X\cup_f D^n, Y\cup_f D^n)$, the only terms in $\partial(c'')|_{n(Y)}$ are these corresponding to $\eta((\partial-\partial_0)c'_e)$.

It follows that $c''$ is almost a chain we are looking for, the only problem is that $c''$ is not a cycle in $C_*^{\lf}(X,Y)$, because it has some boundary outside of $Y$. However, because this boundary can be made arbitrary small, we can use techniques similar to these used in the proof of Proposition \ref{prop_chain_decomposition} to correct $c''$ to be a proper cycle.

Choose some $\varepsilon>0$ and let $T$ be a chain homotopy between $S^{(N)}$ and $Id$. Because $\|T\|$ depends only on $N$, which depends only on $c'$, we can set $\varepsilon'=\frac{\varepsilon}{\|T\|}$. Then $|(\partial c'')|_{n(Y)}|_1\leq \frac{\varepsilon}{\|T\|}$. Consider the chain
\[
\bar{c} = c'' + T((\partial c'')|_{n(Y)}).
\]
The term $(\partial c'')|_{n(Y)}$ has finite support, hence $\bar{c}\in C_*^{\lf}(X)$ is a finite modification of $c$. Moreover, it has the following properties (because adding a term in $C^{\lf}_*(Y)$ does not change a homology class in $C^{\lf}_*(X,Y)$, we will check some of them for $\hat{c} = c'' + T(\partial c'')$).
\begin{itemize}
\item It is a cycle. We have
\[
\partial\hat{c} = \partial c'' +\partial T\partial c'' = \partial c'' + S^{(N)}\partial c'' -\partial c'' -T\partial\partial c'' = \partial S^{(N)}(c'')\in C_*(Y).
\]
\item It is homologuous to $c'$. We have
\begin{eqnarray*}
[\bar{c}] = [\hat{c}] & = & [S^{(N)}\hat{c}] = [S^{(N)}c''' + S^{(N)}T\partial c'''] \\ 
& = & [S^{(N)}c''' + S^{(2N)} c''' - S^{(N)}c''' - \partial S^{(N)}T c'''] = [S^{(2N)}c'''] = [S^{(2N)}(c')] = [c'].
\end{eqnarray*}

\item $|\bar{c}|^{ne(Y')}_1 \leq |c''|^{ne(Y')}_1 + \|T\|\cdot |(\partial c'')|_{n(Y)}|^{ne(Y')}_1 \leq |c'|^{ne(Y')}_1 + \varepsilon$.
\end{itemize}
It follows that $\bar{c}$ satisfies the required conditions.
\end{proof}

\subsection*{Acknowledgements}
I am grateful to Piotr Nowak for bringing the simplicial volume to my attention. I would also like to thank Koji Fujivara, Clara L\"{o}h and Roman Sauer for discussions about the preliminary version of this paper.

The author was supported by NCN grant UMO-2014/13/N/ST1/02421.

\bibliography{references}

\end{document}